\documentclass[11pt, a4paper]{amsart}
\usepackage{amssymb, array, amsmath, amscd, pdfpages, enumerate, amsthm, fixltx2e, setspace, hyperref}

\usepackage{xcolor}
\usepackage{mathtools}

\usepackage{ifthen}
\DeclareMathOperator{\re}{Re}			
\DeclareMathOperator{\im}{Im}			
\DeclareMathOperator{\rank}{rank}		
\newcommand{\ran}{\mathcal{R}}						
\renewcommand{\ker}{\mathcal{N}}					

\DeclareMathOperator*{\dist}{dist}       

\newcommand*{\ldelim}[2]{\csname#1l\endcsname#2}   
\newcommand*{\rdelim}[2]{\csname#1r\endcsname#2}   
\newcommand*{\mdelim}[2]{\csname#1m\endcsname#2}   

\newcommand{\ga}{\alpha}

\newcommand{\gl}{\lambda}

\newcommand{\gs}{\sigma}
\newcommand{\gz}{\zeta}
\newcommand{\eps}{\varepsilon}

\newcommand*{\C}{{\mathbb{C}}}     

\newcommand*{\Z}{{\mathbb{Z}}}     
\newcommand*{\N}{{\mathbb{N}}}     

\newcommand*{\Ran}{{\mathcal{R}}}   

\newcommand*{\abs} [1]{\lvert#1\rvert}
\newcommand*{\norm}[1]{\lVert#1\rVert}

\newcommand*{\set} [1]{\{#1\}}
\newcommand*{\dualpair}[2]{\langle#1,#2\rangle}

\newcommand*{\Abs}[2][default]{\ifthenelse{\equal{#1}{default}}{\left\lvert#2\right\rvert}{\ldelim{#1}{\lvert}#2\rdelim{#1}{\rvert}}}
\newcommand*{\Norm}[2][default]{\ifthenelse{\equal{#1}{default}}{\left\lVert#2\right\rVert}{\ldelim{#1}{\lVert}#2\rdelim{#1}{\rVert}}}

\newcommand*{\Iprod}[3][default]{\ifthenelse{\equal{#1}{default}}{\left\langle#2,#3\right\rangle}{\ldelim{#1}{\langle}#2,#3\rdelim{#1}{\rangle}}}
\newcommand*{\Dualpair}[3][default]{\ifthenelse{\equal{#1}{default}}{\left\langle#2,#3\right\rangle}{\ldelim{#1}{\langle}#2,#3\rdelim{#1}{\rangle}}}

\newcommand*{\Set}[2][default]{\ifthenelse{\equal{#1}{default}}{\left\{#2\right\}}{\ldelim{#1}{\{}#2\rdelim{#1}{\}}}}

\newcommand*{\dda}[3][1]{\ifthenelse{\equal{#1}{1}}{\frac{d#3}{d#2}}{\frac{d^{#1}#3}{d#2^{#1}}}}
\newcommand*{\ddb}[2][1]{\ifthenelse{\equal{#1}{1}}{\frac{d}{d#2}}{\frac{d^{#1}}{d#2^{#1}}}}
\newcommand*{\pd}[3][1]{\ifthenelse{\equal{#1}{1}}{\frac{\partial{#2}}{\partial{#3}}}{\frac{\partial^{#1}{#2}}{\partial#3^{#1}}}}
\newcommand*{\pdb}[2][1]{\ifthenelse{\equal{#1}{1}}{\frac{\partial}{\partial{#2}}}{\frac{\partial^{#1}}{\partial#2^{#1}}}}

\newcommand*{\lp}[1][p]{\ell^{#1}}

\newcommand*{\inv}{^{-1}}

\newcommand{\pmat}[1]{\begin{pmatrix}#1\end{pmatrix}}

\newcommand{\eq}[1]{\begin{align*}#1\end{align*}}
\newcommand{\eqn}[1]{\begin{align}#1\end{align}}

\newcommand{\hiddenproof}[2]{
\ifthenelse{#1}{Hidden}{#2}
}

\newcommand{\longcomm}[1]{}

\usepackage{datetime}
\newdateformat{findate}{\THEDAY.\THEMONTH.\THEYEAR}

\let\ran\relax
\let\Ran\relax
\let\ker\relax
\let\Ran\relax

\DeclareMathOperator*{\Ker}{Ker}
\DeclareMathOperator*{\ker}{Ker}
\DeclareMathOperator*{\Ran}{Ran}
\DeclareMathOperator*{\ran}{Ran}
\DeclareMathOperator*{\Fix}{Fix}

\newcommand{\dd}{\mathrm{d}}
\newcommand{\Mlog}{M_\mathrm{log}}
\newcommand{\mlog}{m_\mathrm{log}}
\newcommand{\B}{\mathcal{B}}
\newcommand{\RR}{\mathbb{R}}
\newcommand{\CC}{\mathbb{C}}
\newcommand{\NN}{\mathbb{N}}
\newcommand{\ZZ}{\mathbb{Z}}

\newcommand{\Ad}{A_0} 
\newcommand{\Aod}{A_1} 
\newcommand{\dr}{\gz}

\newcommand{\Normgg}[1]{\biggl\|#1\biggr\|}

\newtheorem{thm}{Theorem}[section]
\newtheorem{prp}[thm]{Proposition}
\newtheorem{lem}[thm]{Lemma}

\theoremstyle{definition}
\newtheorem{ass}[thm]{Assumptions}
\newtheorem{rem}[thm]{Remark}
\newtheorem{ex}[thm]{Example}

\numberwithin{equation}{section}

\begin{document}

\title[Asymptotics for systems of differential equations]{Asymptotics for infinite systems of differential equations}

\author{Lassi Paunonen}
\address{Department of Mathematics, Tampere University of Technology, PO. Box 553, 33101 Tampere, Finland}
\email{lassi.paunonen@tut.fi}

\author{David Seifert}
\address{St John's College, St Giles, Oxford\;\;OX1 3JP, United Kingdom}
\email{david.seifert@sjc.ox.ac.uk}

\begin{abstract}
This paper investigates the asymptotic behaviour of solutions to certain  infinite systems of ordinary differential equations. In particular, we use results from ergodic theory and the asymptotic theory of $C_0$-semigroups to obtain a characterisation, in terms of convergence of certain Ces\`aro averages, of those initial values which lead to convergent solutions. Moreover, we obtain estimates on the \emph{rate} of convergence for solutions whose initial values satisfy a stronger ergodic condition. These results rely on a detailed spectral analysis of the operator describing the system, which is made possible by certain structural assumptions on the operator. The resulting class of systems is sufficiently broad to cover a number of  important applications, including in particular both the so-called \emph{robot rendezvous problem} and an important class of \emph{platoon systems} arising in control theory. Our method leads to new results in both cases.
\end{abstract}

\thanks{Part of this work was 
carried out
while the first author visited Oxford in March and April 2015. The visit was supported by the EPSRC grant  EP/J010723/1 held by Professor C.J.K.\ Batty (Oxford) and Professor Y.\ Tomilov (Warsaw).
The first author is  funded by the Academy of Finland grant number 298182. Both authors
would like to thank Professor Batty for his useful comments on an earlier version of this paper.}

\maketitle

\section{Introduction}

The purpose of this paper is to study the asymptotic behaviour of solutions to infinite systems of coupled ordinary differential equations. In particular, given $m\in\NN$, we consider time-dependent vectors $x_k(t)$  satisfying 
\eqn{
\label{eq:ODEintro}
\dot{x}_k(t) &= \Ad x_k(t) +\Aod x_{k-1}(t), \quad k\in\Z,\;  t\ge0,
}
for $m\times m$ matrices $\Ad$ and $\Aod$, and we assume that the initial values $x_k(0)\in\CC^m$, $k\in\ZZ$, are known.  The characteristic feature of this class of systems is that the dynamics of each subsystem depend not only on the state of the subsystem itself but also the state of the previous subsystem.  Systems of this type arise naturally in applications, and indeed our investigation of such models is motivated by two  important examples.

The first is the so-called \emph{robot rendezvous problem} \cite{FeiFra12,FeiFra12b}, where $m=1$, $A_0=-1$ and $A_1=1$.
In this case the equations in~\eqref{eq:ODEintro} can be thought of as describing the motion in the complex plane of countably many vehicles, or \emph{robots}, indexed by the integers $k\in\ZZ$, following the rule that robot $k$ moves in the direction of robot $k-1$ with speed equal to their separation.  A second important example in which the general model \eqref{eq:ODEintro} arises is the study of \emph{platoon systems} in control theory; see for instance \cite{PloSch11, SwaHed96,ZwaFir13}. Here we begin with a  more realistic dynamical model of our  vehicles by associating with each a position in the complex plane as well as a velocity and an acceleration, and the control objective is to steer the vehicles towards a state in which, for each $k\in\ZZ$,   vehicle $k$ is a certain target separation $c_k\in\CC$ away from vehicle $k-1$ and all vehicles are moving at a target velocity $v\in\CC$. This model too can be written in the form \eqref{eq:ODEintro} for $m=3$ and suitable $3\times 3$ matrices $A_0$ and $A_1$ which involve certain control parameters that need to be fixed. In both cases the key question is whether solutions converge to a limit as $t\to\infty$. Thus in the robot rendezvous problem we would like to know whether the positions of the robots  converge to a mutual meeting, or \emph{rendezvous}, point, and in the platoon system we ask whether we can choose the control parameters in such a way that the vehicles asymptotically approach their target state. 

We present a unified approach to the study of these problems by first reformulating the system \eqref{eq:ODEintro} as the abstract Cauchy problem
\eqn{
\label{CP} 
\begin{cases}
\dot{x}(t)= Ax(t), &t\ge0,\\
 x(0)=x_0\in X,
\end{cases}}
on the space $X=\lp[p](\C^m)$ with $m\in\NN$ and $1\leq p\le\infty$.
Note that \eqref{CP} indeed becomes \eqref{eq:ODEintro} if we let  $x(t) = (x_k(t))_{k\in\Z}$ for $t\ge0$, $x_0=(x_k(0))_{k\in\ZZ}$  and take the bounded linear operator $A$ to act by sending a sequence $(x_k)_{k\in\ZZ}\in X$ to 
 $$Ax=(A_0x_k+A_1x_{k-1})_{k\in\ZZ}.$$
 Systems of this form are examples of
 what are sometimes called  ``spatially invariant systems'', where in general it is possible for the dynamics of each subsystem to depend on more than just one other subsystem; see for instance \cite{BamPag02}. Our main objective is to investigate
 whether or not 
 the solution $x(t)$, $t\ge0$, of~\eqref{CP} converges to a limit as $t\to\infty$ and, if so, 
 what
 can be said about the
rate of convergence. Most of the existing research into such systems is confined to the Hilbert space case $p=2$. For instance, it is shown in \cite{CurIft09} using Fourier transform techniques that 
solutions $x(t)$, $t\ge0$, of 
 some spatially invariant systems of the form~\eqref{eq:ODEintro}
 on the space $X=\lp[2](\C^2)$
 satisfy $x(t)\to0$ 
 as $t\to\infty$ for all initial values $x_0\in X$,
 but that there exists no
uniform
rate of decay.
 Since the Fourier transform approach is specific to the Hilbert space setting, we develop a  new approach to studying the asymptotic behaviour of solutions of~\eqref{CP} in the case where the matrices $A_0$ and $A_1$ satisfy certain additional assumptions. Specifically, we assume throughout that $A_1\ne0$ to avoid the trivial uncoupled case, but more importantly we suppose that there exists a rational function $\phi$ such that 
  \eqn{
  \label{eq:PSintro}
  \Aod (\gl-\Ad)\inv\Aod
  = \phi(\gl)\Aod, \qquad \gl\in\CC\setminus \gs(\Ad).
  }
 When such a function $\phi$ exists we call it the \emph{characteristic function} of our system.  Both the robot rendezvous problem and the platoon system fall into this special class, as indeed do many other systems. For systems having this property we develop techniques allowing us to handle  the full range $1\le p\le\infty$ rather than just the case $p=2$, and in particular we include the cases $p=1$ and $p=\infty$, where it turns out no longer to be the case that all solutions converge to a limit. In fact our approach, which is based on a detailed  analysis of the operator $A$ and the $C_0$-semigroup it generates,  leads to a complete understanding of which initial values do and which do not lead to convergent solutions in these cases, and moreover gives an estimate on the  \emph{rate} of convergence for a certain subset of initial values.

The paper is organised as follows.
Our main theoretical results are presented in 
Sections~\ref{sec:spec}, \ref{sec:bound}, and~\ref{sec:asymp}.
In Section~\ref{sec:spec} we examine the spectral properties of $A$, and the main results are Theorem~\ref{thm:Aspec}, which among other things provides a very simple characterisation of the set $\sigma(A)\setminus\sigma(A_0)$ in terms of the characteristic function $\phi$, namely
\eq{
\gs(A)\setminus \gs(\Ad) =\big\{\gl\in\C\setminus\gs(A_0):\abs{\phi(\gl)}=1\big\},
}
and Proposition~\ref{prp:Rnormestimates}, which describes the behaviour of the resolvent operator of $A$ in the neighbourhood of spectral points. In Section~\ref{sec:bound} we turn to the delicate issue of whether the semigroup generated by $A$ is uniformly bounded. The main result here is Theorem~\ref{thm:Aunifbdd}, which gives a sufficient condition for uniform boundedness involving the derivatives of $\phi$. In Section~\ref{sec:asymp}, we then combine the results of Sections~\ref{sec:spec} and \ref{sec:bound} with known results in ergodic theory and
recent
results in the theory of $C_0$-semigroups~\cite{BatChi14,ChiSei15, Mar11} in order to obtain our main result, Theorem~\ref{cor:gen}, which describes the asymptotic behaviour of solutions to general systems in our class. For instance, it is a consequence of Theorem~\ref{cor:gen}  that there exists an even integer $n\ge2$ determined solely by the characteristic function $\phi$  such that for all $x_0\in X$ the derivative of the solution $x(t)$, $t\ge0$, of \eqref{CP} satisfies the quantified decay estimate
 \eqn{
 \label{eq:decayintro}
\|\dot{x}(t)\|=O \left(\bigg(\frac{\log t}{t}\bigg)^{1/{n}}\right),\quad t\to\infty,
}  
and the logarithm can be omitted if $p=2$. Moreover, for $1<p<\infty$ not only the derivative of each solution  but also the solution itself decays to zero as $t\to\infty$, but this is no longer true when $p=1$ or $p=\infty$. In these cases, Theorem~\ref{cor:gen} gives a characterisation, in terms of convergence of certain Ces\`{a}ro means, of those initial values $x_0\in X$ which do lead to convergent solutions, and the result also shows that under a supplementary condition the convergence of solutions to their limit can be quantified in a form analogous to \eqref{eq:decayintro}. 

In Sections \ref{sec:plat} and \ref{sec:robots}  we return to the motivating examples. First, in Section~\ref{sec:plat}, we apply the general result in the setting of the platoon system, which leads to extensions of results obtained previously in \cite{CurIft09, ZwaFir13} for the Hilbert space case $p=2$. In particular, the main result in this section, Theorem~\ref{cor:plat}, shows that the platoon system approaches its target for all $x_0\in X$ not just for $p=2$, as was shown in \cite{ZwaFir13}, but more generally when $1<p<\infty$. 
We also show that for $p=1$ and $p=\infty$ this statement is no longer true but our Theorem~\ref{cor:plat} provides a simple ergodic condition on the initial \emph{displacements} of the vehicles which is necessary and sufficient for the solution to converge to a limit. Then, in Section~\ref{sec:robots} we return to the robot rendezvous problem and use our general result, Theorem~\ref{cor:gen}, to settle several questions left open  in \cite{FeiFra12,FeiFra12b}. We conclude in Section~\ref{sec:concl} by mentioning several topics which remain subjects for future research.

The notation we use is more or less standard throughout. In particular, given a complex Banach space $X$, the norm on $X$ will typically be denoted by $\|\cdot\|$ and occasionally, in order to avoid ambiguity, by $\|\cdot\|_X$. In particular, for $m\in\NN$ and $1\le p\le\infty$, we let $\ell^p(\CC^m)$ denote the space of doubly infinite sequences $(x_k)_{k\in\ZZ}$ such that $x_k\in\CC^m$ for all $k\in\ZZ$ and $\sum_{k\in\ZZ}\|x_k\|^p<\infty$ if $1\le p<\infty$ and $\sup_{k\in\ZZ}\|x_k\|<\infty$ if $p=\infty$. Here and in all that follows we endow the finite-dimensional space $\CC^m$ with the standard Euclidean norm and we consider $\ell^p(\CC^m)$ with the norm given for $x=(x_k)_{k\in\ZZ}$ by $\|x\|=(\sum_{k\in\ZZ}\|x_k\|^p)^{1/p}$ if $1\le p<\infty$ and $\|x\|=\sup_{k\in\ZZ}\|x_k\|$ if $p=\infty$. With respect to this norm $\ell^p(\CC^m)$ is a Banach space for $1\le p\le\infty$ and a Hilbert space when $p=2$. We write $X^*$ for the dual space of $X$, and given $\phi\in X^*$ the action of $\phi$ on $x\in X$ is written as $\dualpair{x}{\phi}$. Moreover we write $\B(X)$ for the space of bounded linear operators on $X$, and given $A\in\B(X)$ we write $\ker(A)$ for the kernel and $\ran(A)$ for the range of $A$. Moreover, we let $\sigma(A)$ denote the spectrum of $A$ and, for $\lambda\in\CC\setminus\sigma(A)$ we write $R(\lambda,A)$ for the resolvent operator $(\lambda-A)^{-1}$. We write $\sigma_p(A)$ for the point spectrum and $\sigma_{ap}(A)$ for the approximate point spectrum of $A$. Given $A\in\B(X)$ we denote the dual operator of $A$ by $A'$. If $A$ is a matrix, we write $A^T$ for the transpose of $A$. Given two functions $f$ and $g$ taking values in $(0,\infty)$, we write $f(t)=O(g(t))$, $t\to\infty,$ if there exists a constant $C>0$ such that $f(t)\le Cg(t)$ for all sufficiently large values of $t$. If $f(t)=O(g(t))$ and $g(t)=O(f(t))$ as $t\to\infty$, or more generally as the argument $t$ tends to some point in the extended complex plane, we write $f(t)\asymp g(t)$ in the limit. Given two real-valued quantities $a$ and $b$, we write $a\lesssim b$ if there exists a constant $C>0$ such that $a\le Cb$ for all values of the parameters that are free to vary in a given situation. Finally, we denote the open right/left half plane by $\CC_\pm=\{\lambda\in\CC:\re\lambda\gtrless 0\}$, and we use a horizontal bar over a set to denote its closure.

\section{Spectral theory}
\label{sec:spec}

We begin by stating two standing assumptions on the matrices $A_0$, $A_1$ appearing in \eqref{eq:ODEintro}.

\begin{ass}
  \label{ass:PScond}
  We assume that 
  \eqn{\label{A1}\tag{A1}
  \Aod\neq 0.
  }
  Moreover we assume that there exists a function $\phi$ such that 
  \eqn{\label{A2}\tag{A2}
  \Aod R(\gl,\Ad)\Aod = \phi(\gl)\Aod, \qquad \gl\in\CC\setminus\sigma(\Ad).
  }
  If this assumption is satisfied we call $\phi$ the \emph{characteristic function}.
\end{ass}

\begin{rem}\label{rem:ass1}
It is clear that if (A2) is satisfied then the characteristic function $\phi$ is a rational function whose poles belong to the set $\gs(\Ad)$. Note also that for $|\lambda|>\|A_0\|$ we have
$$|\phi(\lambda)|\|A_1\|\le\frac{\|A_1\|^2}{|\lambda|-\|A_0\|}.$$
In particular,  when (A1) and (A2) both hold it follows that $|\phi(\lambda)|\to0$ as $|\lambda|\to\infty$. It is straightforward to show that both (A1) and (A2) are satisfied whenever $\rank(A_1)=1$.
\end{rem}

In this section we characterise the spectrum of the operator $A$ under our standing assumptions (A1) and (A2). The following is the main result. It essentially characterises the spectrum of $A$ in terms of the characteristic function $\phi$. Here and in what follows we use the notation 
$$\Omega_\phi=\big\{\gl\in\C\setminus\sigma(A_0):\abs{\phi(\gl)}=1\big\}.$$

\begin{thm}
  \label{thm:Aspec}
  Let  $1\le p\le\infty$ and $m\in\NN$, and suppose that \textup{(A1)}, \textup{(A2)} hold. Then the spectrum of $A$ satisfies 
  \eqn{\label{eq:spec_gen}
  \gs(A)\setminus \gs(\Ad) = \Omega_\phi.
  }
  Moreover, the following hold:  
  \begin{itemize}
      \setlength{\itemsep}{.5ex}
    \item[\textup{(a)}]  If $1\leq p<\infty$, then $\gs(A)\setminus \gs(\Ad)\subset \gs_{ap}(A)\setminus \gs_p(A)$.
\item[\textup{(b)}] If $p=\infty$, then $\gs(A)\setminus \gs(\Ad)\subset \gs_p(A)$
  and, given $\lambda\in\gs(A)\setminus\gs(A_0)$,
  \eqn{\label{eq:eigenspace}
  \ker(\gl -A)=\big\{(\phi(\gl)^k x_0)_{k\in\Z}:x_0\in \ran(R(\gl,\Ad)\Aod)\big\}.
  }
  In particular, $\dim\ker(\gl-A)=\rank(\Aod)$ for all $\lambda\in\gs(A)\setminus\gs(A_0)$.
  \end{itemize}
 Furthermore, for $\gl\in\gs(A)\setminus \gs(\Ad)$ the range of $\gl-A$ is dense in $X$ if and only if $1<p<\infty$.
\end{thm}

\begin{rem}
  \label{rem:AspecA0eigs}
  The points $\gs(\Ad)$ may be in either $\gs(A)$ or $\rho(A)$ depending on the matrices $\Ad$ and $\Aod$. Note for instance that, given $\lambda\in\CC$,  any vector $x=(x_k)_{k\in\Z}$ with $x_0\in\ker(\gl-A_0)\cap\ker(A_1)$ and $x_k=0$ for $k\neq 0$ satisfies $x\in \ker(\gl-A)$. In particular, $\gl\in\sigma(A_0)\cap\sigma(A)$ whenever $\ker(\gl-A_0)\cap\ker(A_1)\ne\{0\}$. Moreover, if $\lambda\in\CC$ is such that $\ran(\gl-\Ad)+\ran(\Aod)\neq \C^m$ then
it is easy to see that any sequence $(x_k)_{k\in\ZZ}\in X$ such that $x_k\not\in\ran(\gl-\Ad)+\ran(\Aod)$ for some $k\in\ZZ$ has an open neighbourhood which is disjoint from $\ran(\gl-A)$, so $\ran(\gl-A)$ cannot be dense in $X$ and once again  $\lambda\in\sigma(A_0)\cap\sigma(A)$. In Sections~\ref{sec:plat} and \ref{sec:robots} we will see examples in which, by contrast, we have  $\sigma(A_0)\cap\sigma(A)=\emptyset$.
\end{rem}

\begin{proof}[Proof of Theorem~\textup{\ref{thm:Aspec}}]
 We begin by showing that every $\gl\in\CC\setminus\gs(A_0)$ such that $\abs{\phi(\gl)}\ne1$ belongs to $\rho(A)$. Indeed, given $\gl\in \CC\setminus\gs(\Ad)$,  let $R_\gl = R(\gl,\Ad)$. Supposing first that  $|\phi(\lambda)|<1$, we consider the operator $R(\gl)\in \B(X)$ given by
  \eqn{
  \label{eq:RAformula}
  R(\gl)x
  &= \left( R_\gl x_k + R_\gl \Aod R_\gl \sum_{\ell=0}^\infty \phi(\gl)^\ell x_{k-\ell-1}\right)_{k\in\Z}
  }
 for all $x=(x_k)_{k\in\Z}\in X$, noting that this gives a well-defined element of $X$ by Young's inequality. Using the fact that 
  $(\Aod R_\gl)^\ell = \phi(\gl)^{\ell-1}\Aod R_\gl$ for all $\ell\in\N$ as a consequence of assumption~(A1), it is straightforward to verify that $(\lambda-A)R(\lambda)x=R(\lambda)(\lambda-A)x=x$ for all $x\in X$, and hence $\lambda\in \rho(A)$ and $R(\gl,A)=R(\gl)$.
A completely analogous argument goes through for $\gl\in\CC\setminus\gs(A_0)$ such that $|\phi(\gl)|>1$, with the only difference that the operator $R(\gl)\in \B(X)$ is now defined by 
  \eq{ 
R(\gl)x
= \left( R_\gl x_k - R_\gl \Aod R_\gl\sum_{\ell=0}^\infty \phi(\gl)^{-\ell-1} x_{k+\ell} \right)_{k\in\Z}
  }
  for all $x\in X$.
This shows that $\gs(A)\setminus \gs(\Ad)\subset \Omega_\phi$.

Suppose now that $1\leq p<\infty$ and let $\gl\in\Omega_\phi$. We will first show that $\gl\notin \gs_p(A)$.
  To this end, let $x\in X$ be such that $(\gl-A)x=0$.
 Then a simple calculation shows that $x_k=\phi(\gl)^{k-\ell-1}R_\gl \Aod  x_{\ell}  $  for all $k,\ell\in\ZZ$ with $k>\ell$, and in particular $\|x_k\|=\|R_\gl \Aod  x_{\ell} \|$ for all $k>\ell$. Hence the assumption that $x\in X$ implies that $x=0$ and therefore $\gl\notin \gs_p(A),$ as required.
In order to  show that $\gl\in\gs_{ap}(A)$,
choose $y_0\in \C^m$ such that $\Aod y_0\neq 0$ and, for $n\in\N$, define the sequence $x^n = (x_k^n)_{k\in\ZZ}\in X$ by
$$ x_k^n = \frac{\phi(\gl)^k R_\gl \Aod y_0}{(2n+1)^{1/p}\norm{R_\gl \Aod y_0}},\quad \abs{k}\leq n,$$
and $x_k^n=0$ otherwise.
  Then $\norm{x^n}_p=1$ for all $n\in\N$,
and 
a direct computation shows that
  \eq{
  \norm{ (\gl-A)x^n}^p 
  =\frac{\norm{y_0}^p + \norm{\Aod R_\gl \Aod y_0}^p}{(2n+1)\norm{R_\gl \Aod y_0}^p}
  \rightarrow 0,\quad n\to\infty.
  }
Thus $\gl\in\gs_{ap}(A)$, which establishes (a).

Now suppose that $p=\infty$ and let $\gl\in \Omega_\phi$.
We will prove that \eqref{eq:eigenspace} holds,  from which (b) follows.
Note first that if $x_0\in\ran(R_\gl \Aod)$ 
then
 a simple calculation shows that $(\phi(\gl)^k x_0)_{k\in\Z}\in \ker(\gl-A)$. 
On the other hand,
if $x=(x_k)_{k\in\ZZ}\in \ker(\gl-A)$, then 
$$(\gl-\Ad)x_k-\Aod x_{k-1}=0$$ 
and hence $x_k=R_\gl \Aod x_{k-1}\in \ran(R_\gl A_1)$ for all $k\in\Z$. Since  
$$x_k = (R_\gl \Aod)^2 x_{k-2}
=\phi(\gl) R_\gl \Aod x_{k-2}
=\phi(\gl)x_{k-1},\quad k\in\ZZ,$$
by assumption (A1)
we obtain that $x_k=\phi(\gl)^k x_0$ for all $k\in\ZZ$. Thus (b) follows, and by combining (a) and (b) with the fact that  $\gs(A)\setminus \gs(\Ad)\subset \Omega_\phi$ we obtain \eqref{eq:spec_gen}. It remains to prove the final statement.

Suppose first that $1<p<\infty$ and let $q=p(p-1)^{-1}$ be the H\"older conjugate of $p$. Moreover let $\gl\in\Omega_\phi$ and  that $y=(y_k)_{k\in\ZZ}\in X^\ast=\ell^q(\CC^m)$ is such that $\dualpair{(\gl-A)x}{y}=0$ for all $x\in X$. Then $y\in\ker(\gl-A')$, where the dual operator $A'$ of $A$ is given by $A'y=(A_0^Ty_k+A_1^Ty_{k+1})_{k\in\ZZ}$ for all $y=(y_k)_{k\in\ZZ}\in X^*$.
Since by assumption~(A1) we have  $(\Aod R_\gl \Aod)^T = {\phi(\gl)}\Aod^T,$
a direct computation shows that 
\eq{
y_k= {\phi(\gl)}^{\ell-k-1}R_\gl^T \Aod^T y_{\ell}
}
for all $k,\ell\in\ZZ$ with $k<\ell$. As in the above argument showing that $\gl\notin \gs_p(A)$ we obtain that $y=0$,  and hence  ${\ran(\gl-A)}$ is dense in $X$ by a standard corollary of the Hahn-Banach theorem. On the other hand, if $p=1$ and $\gl\in\Omega_\phi$ we can consider the element $y=(y_k)_{k\in\ZZ}\in X^*=\ell^\infty(\CC^m)$ with entries 
\eq{
y_k={\phi(\gl)}^{-k} R_\gl^T A_1^T y_0,\quad k\in\Z,
}
where $y_0\in\C^m$ is chosen in such a way that $A_1^T y_0\neq 0$. A simple verification shows $y\in\ker(\lambda-A')$ and hence that $\dualpair{(\gl-A)x}{y}=0$ for all $x\in X$, so ${\ran(\gl-A)}$ cannot be dense in $X$. Finally, suppose that $p=\infty$ and that $\gl\in\Omega_\phi$. Let $y=(\phi(\gl)^ky_0)_{k\in\Z}\in X$, where $y_0\in\C^m$ is such that $R_\gl\Aod R_\gl y_0\neq 0$. We show that $y$ lies outside the closure of ${\ran(\gl-A)}$. Indeed, let  $0<\eps <\norm{R_\gl \Aod R_\gl y_0}/\norm{R_\gl \Aod R_\gl}$ and suppose for the sake of contradiction that there exists $x\in X$ such that
\eq{
\norm{(\gl-A)x-y} = \sup_{k\in\Z} \; \norm{(\gl-\Ad)x_k-\Aod x_{k-1}-y_k}< \eps.
}
Let $z_k=(\gl-\Ad)x_k-\Aod x_{k-1}-y_k $, so that $\|z_k\|<\varepsilon$  for all $k\in\ZZ$.
A simple inductive argument shows that for all $n\in\N$ we have
\eq{
x_0 = \phi(\gl)^{n-1}R_\gl A_1x_{-n}+R_\lambda (y_0+z_0)+R_\gl A_1 R_\gl\sum_{\ell=1}^{n-1}\phi(\gl)^{\ell-1}(y_{-\ell}+z_{-\ell}).
}
Since   
\eq{
\biggl\|R_\gl \Aod R_\gl \sum_{\ell=1}^{n} \phi(\gl)^{\ell-1} y_{-\ell}\biggr\|
= n\norm{R_\gl \Aod R_\gl y_0},\quad n\in\NN,
}
we obtain that
\eq{
\norm{x_0}\geq n\big(\norm{R_\gl \Aod R_\gl y_0}-\eps \norm{R_\gl\Aod R_\gl}\big) - \norm{R_\gl }(\|y_0\|+ \eps )  -\norm{R_\gl \Aod} \norm{x}
}
for all $n\in\NN$. However, by the choice of $\varepsilon$ this is absurd. Hence no such $x\in X$ exists and in particular the range of $\gl-A$ is not dense in $X$. This completes the proof.
\end{proof}

The next result establishes a useful estimate for the norm of the resolvent operator in the neighbourhood of singular points.

\begin{prp}
  \label{prp:Rnormestimates}
Fix $1\le p\le\infty$ and $m\in\NN$, and suppose that \textup{(A1)}, \textup{(A2)} hold. If $\gl\in\CC\setminus \gs(\Ad)$ is such that $\abs{\phi(\gl)}\neq 1$, then
  \eq{
  \left|\|R(\lambda,A)\|-\frac{\|R(\lambda,A_0)A_1R(\lambda,A_0)\|}{|1-|\phi(\gl)||}\right|\le\|R(\lambda,A_0)\|.
  }
  In particular, for $\gl_0\in\CC\setminus\gs(A_0)$ such that $|\phi(\lambda_0)|=1$ we have 
  \eq{
 \norm{R(\gl,A)}
\asymp \frac{1}{\abs{1- \abs{\phi(\gl)}}}
  }
  as $\gl\to\gl_0$ in the region $\{\gl\in\CC\setminus\gs(A_0):\abs{\phi(\gl)}\neq 1\}$.
\end{prp}

\begin{proof}
As in the proof of Theorem~\ref{thm:Aspec}, we let $R_\lambda=R(\gl,A_0)$ for $\gl\in\CC\setminus\gs(A_0)$. We consider the case where $0<\abs{\phi(\gl)}<1$; the case $|\phi(\gl)|>1$ follows similarly, as in the proof of Theorem~\ref{thm:Aspec}. From~\eqref{eq:RAformula}  we see that for $\gl\in\CC\setminus\gs(A_0)$ such that $\abs{\phi(\gl)}<1$ we have 
  $R(\gl,A) = D(\gl) + Q(\gl)$, where $D(\gl) x= (R_\gl x_k)_{k\in\ZZ}$  and 
  \eq{
  Q(\gl)x = \left(R_\gl\Aod R_\gl\sum_{\ell=0}^\infty \phi(\gl)^\ell  x_{k-\ell-1}\right)_{k\in\ZZ}
  }
  for all $x = (x_k)_{k\in\Z}\in X$. Note that $\norm{D(\gl)}= \norm{R_\gl}$, so the result will follow from the triangle inequality once we have established that
  \eqn{\label{eq:Qest}
  \norm{Q(\gl)}= \frac{\norm{R_\gl\Aod R_\gl}}{1- \abs{\phi(\gl)}}.
  }
In fact, since 
$$\norm{Q(\gl)}\leq \frac{\norm{R_\gl\Aod R_\gl}}{1- \abs{\phi(\gl)}}$$ 
for $1\leq p\leq \infty$ by a straightforward estimate, it suffices to prove the converse inequality.

Suppose first that  $p=\infty$ and consider the sequence $x=(e^{ik\theta} y_0)_{k\in\Z}\in X$, where $\theta=\arg \phi(\gl)$ and $y_0\in\C^m$ is such that $\norm{x_0}=1$ and $\norm{R_\gl \Aod R_\gl y_0} = \norm{R_\gl \Aod R_\gl}$. Then $\norm{x}=1$ and
  \eq{
  \norm{Q(\gl)x} 
  &= \sup_{k\in\Z}\; \Normgg{ R_\gl \Aod R_\gl\sum_{\ell=0}^\infty \phi(\gl)^\ell  x_{k-\ell-1}}
  = \frac{\norm{ R_\gl \Aod R_\gl }}{1-\abs{\phi(\gl)}},
  }
thus establishing \eqref{eq:Qest}. Now suppose that $1\leq p<\infty$. Once again let $\theta = \arg\phi(\gl)$ and let $y_0\in\C^m$ be such that $\norm{y_0}=1$ and $\norm{R_\gl \Aod R_\gl y_0} = \norm{R_\gl \Aod R_\gl}$. 
Furthermore, let $\eps\in(0,1)$ and let $M,N\in\N$ be such that 
$$\sum_{\ell=M+1}^\infty \abs{\phi(\gl)}^\ell<\eps\quad\mbox{and}\quad\frac{N-M}{N}>(1-\eps)^p.$$
 Consider the sequence $x = (x_k)_{k\in\Z} \in X$ with entries $x_k=e^{i k \theta}\alpha_ky_0$, where $\alpha_k =N^{-1/p}$ for $ -N\leq k\leq -1$ and $\alpha_k=0$ otherwise. Then $\norm{x}=1$ and 
\eq{
\norm{Q(\gl)x}^p
= \norm{R_\gl\Aod R_\gl}^p  \sum_{k\in\Z}  \left(\sum_{\ell=0}^\infty \abs{\phi(\gl)}^\ell  \ga_{k-\ell-1}\right)^p,
}
and hence by our choices of $M$ and $N$ we obtain that
\eq{
\norm{Q(\gl)x}&\ge\frac{ \norm{R_\gl\Aod R_\gl} }{N^{1/p}}\left(\sum_{M-N+1\le k\le 0} \left(\sum_{\ell=0}^{M} \abs{\phi(\gl)}^\ell  \right)^p\right)^{1/p}\\&>(1-\eps)\frac{\norm{R_\gl\Aod R_\gl} }{1-|\phi(\lambda)|}-\eps(1-\eps)\norm{R_\gl\Aod R_\gl}.
}
Since $\eps\in(0,1)$ was arbitrary, \eqref{eq:Qest} follows and the proof is complete.
\end{proof}

We conclude this section with a  refinement of  Proposition~\ref{prp:Rnormestimates} in an important special case.

\begin{lem}\label{lem:even}
  Fix $1\le p\le\infty$ and $m\in\NN$, and suppose that \textup{(A1)}, \textup{(A2)} hold, and that $0\in\Omega_\phi\subset\CC_-\cup\{0\}$.   Then
  there exists an even integer $n\ge2$ such that $
  1-\abs{\phi(is)}\asymp |s|^n$
  as $\abs{s}\to 0$.
\end{lem}

\begin{proof}
  The rational function $\phi$ is of the form $\phi(\gl)=p(\gl)/q(\gl)$, where $p$ and $q$ are coprime polynomials and the roots of $q$ are contained in the set $\gs(\Ad)\subset \C_-$. 
Since $\abs{\phi(0)}=1$ and $|\phi(\lambda)|\to0$ as $|\lambda|\to\infty$, we have that $\abs{\phi(is)}<1$ for $s\ne0$ and hence
  \eq{
  1-\abs{\phi(is)}
  = \frac{\abs{q(is)}^2 - \abs{p(is)}^2}{\abs{q(is)}(\abs{p(is)}+\abs{q(is)})},\quad s\ne0.
  }
  The denominator of the right-hand side
  is bounded from above and from below near $s=0$. Thus the rate at which $1-\abs{\phi(is)}\to 0$ is equal to that at which $r(s)=\abs{q(is)}^2 - \abs{p(is)}^2\to 0$ as $\abs{s}\to 0$. Since $r$ is a real polynomial satisfying $r(0)=0$ and $r(s)>0$ for $s\ne0$, we  have that $r(s)=s^n r_0(s)$, $s\in\RR$, where $n\in \N$ is even and $r_0$ is a polynomial satisfying $r_0(0)> 0$. The claim now follows.
\end{proof}

\begin{rem}\label{rem:nphi}
Note that $n=n_\phi$ is determined by the characteristic function $\phi$. We call $n_\phi$ the \emph{resolvent growth parameter}.
\end{rem}

\section{Uniform boundedness of the semigroup}
\label{sec:bound}

Consider our general model and assume that assumptions (A1) and (A2) are satisfied. In this section we present conditions 
on the characteristic function $\phi$ under which the semigroup generated by $A$ is uniformly bounded or even contractive.
Since uniform boundedness necessarily requires that $\gs(A)\subset \overline{\C_-}$,  Theorem~\ref{thm:Aspec} shows that it is necessary to assume that $\Omega_\phi \subset \overline{\C_-}$, where $\Omega_\phi=\{\lambda\in\CC\setminus \gs(A_0):|\phi(\gl)|=1\}$. Note also that, since $|\phi(\lambda)|\to0$ as $|\lambda|\to\infty$  by  Remark~\ref{rem:ass1}, we must have $\abs{\phi(\gl)}<1$ for all $\gl\in\C_+$ in this case. The following theorem is the main result of this section.

\begin{thm}
  \label{thm:Aunifbdd}
Let $1\le p\le\infty$ and $m\in\NN$. Suppose that assumptions \textup{(A1)}, \textup{(A2)} hold, that $\sigma(A_0)\subset\CC_-$ and that $\Omega_\phi \subset \overline{\C_-}$.
If  furthermore  
\eqn{
\label{eq:unifbddabscondALTphi}
\sup_{0<\gl\leq 1}\frac{\gl}{1-\abs{\phi(\gl)}}<\infty
\quad\mbox{and}\quad
  \sup_{n\in\N}\, \sup_{\gl>0}\; \frac{\gl^{n+1}}{n!} \sum_{\ell=1}^\infty    \Abs{\ddb[n]{\gl}\phi(\gl)^\ell} <\infty,
  }
  then the semigroup generated by $A$ is uniformly bounded. 
If 
  \eqn{
  \label{eq:Acontrcondition}
  \sup_{\lambda>0} \left(\gl \norm{R(\gl,\Ad)} + \gl \frac{\norm{R(\gl,\Ad)\Aod R(\gl,\Ad)}}{1-\abs{\phi(\gl)}}\right) \leq 1,
  }
  then the semigroup generated by $A$ is contractive.  
\end{thm}

\begin{proof}
Both parts of the result are consequences of the Hille-Yosida theorem. We thus aim to establish a uniform upper bound for $\|\gl^n R(\gl,A)^n\|$ as $\gl>0$ and $n\in\NN$ are allowed to vary. 
For $\lambda>0$ we let $R_\gl=R(\gl,A_0)$.   Then by \eqref{eq:RAformula} in the proof of Theorem~\ref{thm:Aspec} and by standard properties of resolvent operators we have that
\eq{
R(\gl,A)^{n} x
= \left( R_\gl^{n} x_k \right)_{k\in\Z}
+ \left( \frac{(-1)^{n-1}}{(n-1)!}\sum_{\ell=0}^\infty \ddb[n-1]{\gl} \big(\phi(\gl)^\ell R_\gl \Aod R_\gl\big)x_{k-\ell-1} \right)_{k\in\Z}
}
for all $x=(x_k)_{k\in\Z}\in X$, and hence
\eqn{\label{eq:res_bound}
\|\gl^{n}R(\gl,A)^{n}\|\le \|\gl^{n}R_\gl^{n}\|+\frac{\gl^{n}}{(n-1)!}\sum_{\ell=0}^\infty \left\|\ddb[n-1]{\gl} \big(\phi(\gl)^\ell R_\gl \Aod R_\gl\big)\right\|}
 for  all $\gl>0$ and all $n\in\NN$. Now since $\sigma(A_0)\subset\CC_-$, there exists $\varepsilon>0$ such that $A_0+\varepsilon$ generates a uniformly bounded semigroup, and in particular 
\eqn{\label{eq:A0bound}
\sup_{n\in\NN}\sup_{\gl>0}\|(\gl+\eps)^nR_\gl^n\|<\infty.
}
Thus the first term on the right-hand side of \eqref{eq:res_bound} is uniformly bounded as $\gl>0$ and $n\in\NN$ are allowed to vary. It remains to consider the second term. Let $\phi_\ell(\lambda)=\phi(\lambda)^\ell$ and observe that, for $\gl>0$ and $\ell,n\in\ZZ_+$, 
\eq{
\frac{1}{n!} \ddb[n]{\gl} \big(\phi(\gl)^\ell R_\gl \Aod R_\gl\big)
=\sum_{k=0}^{n} \frac{\phi_\ell^{(k)}(\gl)}{k!} \frac{1}{(n-k)!} \ddb[n-k]{\gl}(R_\gl\Aod R_\gl)
}
and  by \eqref{eq:A0bound}
\eq{
\left\|\frac{1}{n!} \ddb[n]{\gl}(R_\gl\Aod R_\gl) \right\|
= \bigg\|\sum_{j=0}^{n}  R_\gl^{j+1} \Aod  R_\gl^{n-j+1}\bigg\|\lesssim \frac{n+1}{(\gl+\varepsilon)^{n+2}}
}
 for all $\gl>0$ and $n\in\ZZ_+$. It follows that 
 \eq{
 \frac{\gl^{n}}{(n-1)!}\sum_{\ell=0}^\infty \left\|\ddb[n-1]{\gl} \big(\phi(\gl)^\ell R_\gl \Aod R_\gl\big)\right\|\lesssim 
\gl^{n}\sum_{k=0}^{n-1} \sum_{\ell=0}^\infty 
\frac{\abs{\phi_\ell^{(k)}(\gl)}}{k!} \frac{n-k}{(\gl+\eps)^{n-k+1}}
 }
 for all $\gl>0$ and $n\in\NN$. Using the first part of assumption \eqref{eq:unifbddabscondALTphi} for the interval $0<\gl\le1$ and the fact that $\sup_{\gl>1}|\phi(\gl)|<1$ for the interval $1<\gl<\infty$, 
 it is straightforward to see that the first term on the right-hand side, corresponding to $k=0$, is uniformly bounded above by
 $$\sup_{n\in\N}\, \sup_{\gl>0}\; \frac{1}{1-|\phi(\gl)|}\frac{n\gl^n}{(\gl+\eps)^{n+1}}<\infty.$$
Using the second part of assumption \eqref{eq:unifbddabscondALTphi} the remaining terms on the right-hand side can be estimated, for all $\lambda>0$ and $n\in\NN$, by
\eq{
\gl^{n}\sum_{k=1}^{n-1} \sum_{\ell=0}^\infty 
\frac{\abs{\phi_\ell^{(k)}(\gl)}}{k!} \frac{n-k}{(\gl+\eps)^{n-k+1}}\lesssim \sum_{k=1}^{n-1}
 \frac{k\gl^{k-1}}{(\gl+\eps)^{k+1}} \le\eps^{-2}.
}
Combining the last two estimates with \eqref{eq:A0bound} in \eqref{eq:res_bound} shows that 
$$\sup_{n\in\N}\, \sup_{\gl>0}\; \|\gl^n R(\gl,A)^n\|<\infty,$$
and hence the semigroup generated by $A$ is uniformly bounded by the Hille-Yosida theorem.

For the second statement we note that if~\eqref{eq:Acontrcondition} holds, then Proposition~\ref{prp:Rnormestimates} shows that 
  $\gl\norm{R(\gl,A)}\leq 1$ for all $\gl>0$, and thus the semigroup generated by $A$ is contractive by the Hille-Yosida theorem.
\end{proof}

The next lemma shows that the assumptions in \eqref{eq:unifbddabscondALTphi} are satisfied in a simple but important special case.

\begin{lem}
  \label{lem:repeated}
Let $\zeta>0$ and $k\in\NN$ be given, and suppose that
$$\phi(\lambda)=\frac{\zeta^k}{(\lambda+\zeta)^k},\quad \lambda\in\CC\setminus\{-\zeta\}.$$
Then both conditions in \eqref{eq:unifbddabscondALTphi} are satisfied.
\end{lem}

\begin{proof}
Note first that 
$$\sup_{0<\gl\leq 1}\frac{\gl}{1-\abs{\phi(\gl)}}=\sup_{0<\gl\leq 1}\frac{\gl(\lambda+\zeta)^k}{(\lambda+\zeta)^k-\zeta^k}\le \sup_{0<\gl\leq 1}\frac{(\lambda+\zeta)^k}{k\zeta^{k-1}}<\infty,$$
so the first part of \eqref{eq:unifbddabscondALTphi} certainly holds. For $n,\ell\in\NN$ and $\lambda>0$ we have that
  \eq{
  \Abs{\ddb[n]{\gl} \phi(\gl)^\ell} 
  ={\dr}^{k\ell} \frac{ (k\ell+n-1)!}{(k\ell-1)!\abs{\gl+\dr}^{k\ell+n}} .
  }
Given $\lambda>0$ let $z = (\gl+\dr)/{\dr}$. Then $z>1$ and 
\eq{
 \sum_{\ell=1}^\infty \Abs{\ddb[n]{\gl} \phi(\gl)^\ell} 
  \le \frac{1}{{\dr}^n} \sum_{\ell=1}^\infty  \frac{ (k\ell+n-1)!}{(k\ell-1)!z^{k\ell+n}}
  \leq \frac{1}{{\dr}^n} \sum_{\ell=1}^\infty  \frac{(\ell+n-1)!}{(\ell-1)!z^{\ell+n}}.
  }
Since 
$$ \sum_{\ell=1}^\infty  \frac{(\ell+n-1)!}{(\ell-1)!z^{\ell+n}} = \sum_{\ell=1}^\infty (-1)^n  \ddb[n]{z}\frac{1}{z^\ell}
  =   \ddb[n]{z}\left( \frac{1}{z-1} \right)= \frac{n!}{(z-1)^{n+1}}$$
and $z-1=\gl/\zeta$,  we obtain that
  \eq{\sup_{n\in\NN}\sup_{\lambda>0}  \frac{\lambda^{n+1}}{n!}\sum_{\ell=1}^\infty \Abs{\ddb[n]{\gl} \phi(\gl)^\ell} \le\zeta,
}
and hence $\phi$ also satisfies the second part of \eqref{eq:unifbddabscondALTphi}, as required.
\end{proof}

\section{Asymptotic behaviour}
\label{sec:asymp}

We now turn to the asymptotic behaviour of solutions to our system \eqref{CP}. For this we require, in addition to our earlier assumptions (A1) and (A2), three further assumptions. Recall that $\Omega_\phi=\{\lambda\in\CC\setminus\sigma(A_0):|\phi(\lambda)|=1\}$, where $\phi$ is the characteristic function of our system.
\begin{ass}
  \label{ass:further}
  We  introduce the further assumptions that 
  \eqn{\tag{A3}\label{A3}
\sigma(A_0)\subset\CC_-,
  } 
  \eqn{\label{A4}\tag{A4}
 0\in \Omega_\phi\subset\CC_-\cup\{0\}\quad\mbox{and}\quad\phi'(0)\ne0,
  }
  \eqn{\tag{A5}\label{A5}
\sup_{t\ge0}\|T(t)\|<\infty,
  }
  where  $T$ is the semigroup generated by $A$.
\end{ass}

\begin{rem}\label{rem:L}
Differentiating the identity in assumption (A2) gives 
$$-A_1R(\lambda,A_0)^2A_1=\phi'(\lambda)A_1,\quad \lambda\in\CC\setminus\sigma(A_0).$$
In particular, if (A3) holds then $-A_1A_0^{-2}A_1=\phi'(0)A_1$, and now assumption  (A4) implies that $A_1A_0^{-1}$ restricts to  an isomorphism from $\Ran(A_0^{-1}A_1)$ onto $\Ran(A_1)$. 
\end{rem}

In what follows we write $L$ for the inverse of this isomorphism appearing in Remark~\ref{rem:L}, so that $L$ maps $\Ran(A_1)$ isomorphically onto $\Ran(A_0^{-1}A_1)$. Moreover, having fixed $1\le p\le\infty$ and $m\in\NN$, we let
\eqn{
\label{stable}
Y=\left\{x_0\in X:\lim_{t\to\infty} x(t)\;\mbox{exists}\right\},
}
where $x(t)$, $t\ge0$, is the solution of \eqref{CP} with initial condition $x(0)=x_0$. Furthermore, we denote the right-shift operator on $X$ by $S$, so that $Sx=(x_{k-1})_{k\in\ZZ}$ for all $x=(x_k)_{k\in\ZZ}\in X$. Recall finally  that $n_\phi$ denotes the resolvent growth parameter of our system; see Remark~\ref{rem:nphi}. The aim in this section is to prove the following theorem. 

\begin{thm}\label{cor:gen}
Let $1\le p\le\infty$, $m\in\NN$ and assume that \textup{(A1)}--\textup{(A5)} hold.  Define the operator $M\in\B(X)$ by $M(x_k)=(A_1 A_0^{-1} x_k),$ and let the operator $L$ and the space $Y$  be defined as above.
\begin{enumerate}
\item[\textup{(a)}]\label{a}  We have $Y=X$  if and only if $1<p<\infty$. More specifically:
\begin{enumerate}[(i)]
\item[\textup{(i)}] If $1<p<\infty$ then $Y=X$ and $x(t)\to0$ as $t\to\infty$ for all $x_0\in X$. 
\item[\textup{(ii)}] If $p=1$ and $x_0\in X$ then $x_0\in Y$ if and only if  
\eqn{
\label{lim_model_finite}
\bigg\|\frac{1}{n}\sum_{k=1}^n\phi(0)^{k}S^{k}Mx_0\bigg\|\to0,\quad n\to\infty,
}
and if this holds then $x(t)\to0$ as $t\to\infty$. 
\item[\textup{(iii)}]If $p=\infty$ and $x_0\in X$ then $x_0\in Y$ if and only if there exists  $y_0\in\Ran(A_1)$ such that for $y=(\phi(0)^ky_0)$ we have 
\eqn{
\label{lim_model_infty}
\bigg\|\frac{1}{n}\sum_{k=1}^n\phi(0)^{k}S^{k}Mx_0-y\bigg\|\to0,\quad n\to\infty,
}
and if this holds then $x(t)\to z$ as $t\to\infty$, where $z=(\phi(0)^kLy_0)$. 
\end{enumerate}

 \item[\textup{(b)}]\label{b} Let $n_\phi$ be the resolvent growth parameter of the system.
 \begin{enumerate}[(i)]
 \item[\textup{(i)}] If $1\le p<\infty$ and the decay in \eqref{lim_model_finite} is like $O(n^{-1})$ as $n\to\infty$ then 
\eqn{
\label{model_log_finite}
\|x(t)\|=O\left(\bigg(\frac{(\log t)^{|1-2/p|}}{t}\bigg)^{1/n_\phi}\right),\quad t\to\infty.
}
\item[\textup{(ii)}] If $p=\infty$ and the decay in \eqref{lim_model_infty} is like $O(n^{-1})$ as $n\to\infty$ then
\eq{
\|x(t)-z\|=O\left(\bigg(\frac{\log t}{t}\bigg)^{1/n_\phi}\right),\quad t\to\infty.
}
\end{enumerate}
\item[\textup{(c)}]\label{c} For  $1\le p\le\infty$ and all  $x_0\in X$ we have
 \eq{
\|\dot{x}(t)\|=O\left(\bigg(\frac{(\log t)^{|1-2/p|}}{t}\bigg)^{1/n_\phi}\right),\quad t\to\infty.
} 
 \end{enumerate}
 \end{thm}

The proof of Theorem~\ref{cor:gen} is based on a number of general results. Given a $C_0$-semigroup $T$  on a complex Banach space $X$, let 
\eqn{
\label{stable_sg}
Y=\left\{x\in X:\lim_{t\to\infty}T(t)x\;\mbox{exists}\right\},
}
noting that this notation is consistent with \eqref{stable}.

\begin{prp}\label{stable_sum}
Let $T$  be a uniformly bounded $C_0$-semigroup on a complex Banach space $X$ and suppose that the generator $A$ of $T$  satisfies $\sigma(A)\cap i\RR=\{0\}$. Then the set $Y$ defined in \eqref{stable_sg} satisfies $Y=X_0\oplus X_1$, where $X_0=\Ker(A)$ and $X_1$ denotes the closure of $\Ran(A)$. Moreover, if $x\in Y$ and $T(t)x\to y$ as $t\to\infty$, then $y=Px$, where $P\in\B(Y)$ is the projection onto $X_0$ along $X_1$.
\end{prp}

\begin{proof}
If $x\in X_0$, then $T(t)x=x$ for all $t\ge0$ and hence $x\in Y$. Thus $X_0\subset Y$. Now define the function $f\in L^1(\RR_+)$ by $f(t)=(t-1)e^{-t}$, then the Laplace transform $F$ of $f$ is given by
$$F(\lambda)=-\frac{\lambda}{(1+\lambda)^2},\quad \re{\lambda}\ge0,$$
and we can define the operator $Q\in\B(X)$ by
$$Qx=\int_0^\infty f(t)T(t)x\,\dd t,\quad x\in X,$$
noting that $Q=AR(1,A)^2$.  Since $F$ vanishes on the set $\sigma(A)\cap i\RR=\{0\}$ and since singleton sets are of spectral synthesis, it follows from the Katznelson-Tzafriri theorem \cite[Theorem~3.2]{Vu92} that $\|T(t)Q\|\to0$ as $t\to\infty$, and hence $\Ran(Q)\subset Y$. A simple argument shows that $\Ran(Q)=\Ran(A)\cap D(A)$, where $D(A)$ denotes the domain of $A$. In particular, $\Ran(Q)$ is dense in $X_1$, so by uniform boundedness of $T$  we obtain that $X_1\subset Y$. Thus $X_0+X_1\subset Y$.  Next we show that the sum is direct. Suppose that $x\in X_0\cap X_1$. Since $x\in X_0$, $\|x\|=\|T(t)x\|$ for all $t\ge0$. On the other hand, since $x\in X_1$, $\|T(t)x\|\to0$ as $t\to\infty$. It follows that $x=0$, and hence $X_0\cap X_1=\{0\}$, as required.

Now suppose that $x\in Y$. Then there exists $y\in X$ such that $y=\lim_{s\to\infty}T(s)x$. For  $t\geq 0$ we have $T(t)y=\lim_{s\to\infty}T(t)T(s)x=y$, which implies that $y\in X_0$.  Let $z=x-y$. Then 
\eqn{
\label{zero}
\|T(t)z\|=\|T(t)x-y\|\to0,\quad t\to\infty.
}
Suppose that $z\in X\setminus X_1$. It follows from a standard application of the Hahn-Banach theorem that there exists $\phi\in X^*$ such that $\dualpair{z}{\phi}=1$ and $\phi|_{X_1}=0$. In particular, $\phi|_{\ran(A)}=0$ and hence $\phi\in \Ker(A')$. It follows that 
$T(t)'\phi=\phi$ for all $t\ge0$, and therefore 
$$\dualpair{T(t)z}{\phi}=\dualpair{z}{T(t)'\phi}=\dualpair{z}{\phi}=1,\quad t\ge0.$$
This contradicts \eqref{zero}, so $z\in X_1$. Thus $x=y+z\in X_0+ X_1$ and consequently $Y= X_0\oplus X_1$. Furthermore, 
$$\lim_{t\to\infty}T(t)x=y=Px,$$
 where $P:Y\to Y$ is the projection onto $X_0$ along $X_1$. Since both $X_0$ and $X_1$ are closed, 
 $P$ is bounded and the proof is complete.
\end{proof}

\begin{rem}\label{rem:ergodic}
Note that if $x\in Y$ and $T(t)x\to y$ as $t\to\infty$, then
\eqn{
\label{eq:ergodic}
\lim_{t\to\infty}\frac{1}{t}\int_0^tT(s)x\,\dd s=y.
}
It is well known that the set of $x\in X$ for which \eqref{eq:ergodic} holds is given by $X_0\oplus X_1$, where $X_0$ and $X_1$ are as in Proposition~\ref{stable_sum}. The result is therefore Tauberian in flavour, showing as it does that \eqref{eq:ergodic} implies  $\lim_{t\to\infty}T(t)x=y$. Note that this ergodic approach also leads to the further equivalent characterisation of the set $Y$ as
$$Y=\left\{x\in X:\lim_{\lambda\to0+}R(\lambda,A)x\;\mbox{exists}\right\}.$$
For details of the above results see for instance \cite[Section~4.3]{ABHN11}, and for a more general result related to Proposition~\ref{stable_sum} see  \cite[Theorem~5.5.4]{ABHN11}.
\end{rem}

 As observed in Remark~\ref{rem:ergodic}, the characterisation of the set $Y$ obtained in Proposition~\ref{ergodic} can be interpreted as the set of mean ergodic vectors of the semigroup $T$ . We now collect some important facts about the set of mean ergodic vectors of certain bounded linear operators, which will then be used to obtain descriptions of the set $Y$ in the case where the semigroup $T$  has a suitable bounded generator.

\begin{prp}\label{ergodic}
Let $X$ be the dual space of a complex Banach space $X_*$ and consider the operator $A=B-C$, where $B,C\in\B(X)$. Suppose that $C$ is invertible and that the operator $Q=C^{-1}B$ is power-bounded and satisfies $Q=U'$ for some $U\in\B(X_*)$. Let $Z= X_0\oplus X_1$, where $X_0=\Ker(A)$ and $X_1$ denotes the closure of $\Ran(A)$, and let $Z_0=X_0\oplus\Ran(A)$.  Then, given $x\in X$, we have $x\in Z$ if and only if there exists $y\in X_0$ such that
\eqn{
\label{proj_lim}
\bigg\|\frac{1}{n}\sum_{k=1}^nQ^kC^{-1}(x-y)\bigg\|\to0,\quad n\to\infty.
} 
Furthermore, $Z_0$ consists of all those $x\in Z$ for which the convergence in \eqref{proj_lim} is  like $O(n^{-1})$ as $n\to\infty$.
\end{prp}

\begin{proof}
Note first that $X_0=\Fix(Q)$ and that $\Ran(A)=\{Cx:x\in\Ran(I-Q)\}$. Hence $X_1=\{Cx:x\in X_2\}$, where $X_2$ denotes the closure of $\Ran(I-Q)$. By \cite[Theorem~1.3 of Section~2.1]{Kre85} and power-boundedness of $Q$,
$$X_2=\left\{x\in X:\lim_{n\to\infty}\frac{1}{n}\sum_{k=1}^nQ^kx=0\right\}.$$
Hence, given $x\in X$, we have $x\in Z$ if and only if there exists $y\in X_0$ such that $x-y\in X_1$, which is equivalent to $C^{-1}(x-y)\in X_2$. This shows that \eqref{proj_lim} holds. The characterisation of $Z_0$ follows similarly using \cite[Theorem~5]{LinSine83}, and it is here that the duality assumptions are needed.
\end{proof}

\begin{rem}\label{sublinear}
 The characterisation of the space $Z$ in fact holds on arbitrary complex Banach spaces and when the condition of power-boundedness  is replaced by the  weaker assumptions that $Q$ is \emph{Ces\`aro bounded}, which is to say
$$\sup_{n\ge1}\bigg\|\frac{1}{n}\sum_{k=1}^nQ^k\bigg\|<\infty,$$ 
and that $\|Q^nx\|=o(n)$ as $n\to\infty$ for each $x\in X$. 
A characterisation of $Z_0$ in this more general setting can be deduced from the results in  \cite{LinSine83}.
\end{rem}

We now seek to combine Propositions~\ref{stable_sum} and \ref{ergodic} in to obtain a characterisation of the set $Y$ defined in \eqref{stable_sg} when the semigroup $T$  is generated by a suitable bounded operator $A$. In particular, we hope to deduce from Proposition~\ref{ergodic} a statement about the \emph{rate} at which certain semigroup orbits converge to  a limit. This requires two abstract results.

 \begin{thm}\label{mlog_gen}
Let $X$ be a complex Banach space and suppose $T$ is a uniformly bounded $C_0$-semigroup on $X$ whose generator $A\in\B(X)$ satisfies $\sigma(A)\cap i\RR=\{0\}$. Suppose that
$$\|R(is,A)\|\le m(|s|),\quad 0<|s|\le1,$$
for some continuous non-increasing function $m:(0,1]\to[1,\infty)$. Then for any $c\in (0,1)$
\begin{equation}\label{eq:mlog}
\|AT(t)\|=O\big(\mlog^{-1}(ct)\big),\quad t\to\infty,
\end{equation}
where $\mlog^{-1}$ is the inverse function of the map $\mlog:(0,1]\to(0,\infty)$ given by
\eqn{
\label{eq:mlogdef}
\mlog(r)=m(r)\log\left(1+\frac{m(r)}{r}\right),\quad 0<r\le1.
}
\end{thm}

\begin{proof}
The result is a consequence of  \cite[Corollary~2.12]{ChiSei15}. Indeed, since $T$ is norm-continuous and hence differentiable, it follows from \cite[Theorem~5.6]{BaBlSr03} that the non-analytic growth bound $\zeta(T)$ of $T$  satisfies $\zeta(T)=-\infty$, and in particular $\zeta(T)<0$. Thus \cite[Corollary~2.12]{ChiSei15} shows that 
$$\|T(t)AR(1,A)\|=O\big(\mlog^{-1}(ct)\big),\quad t\to\infty,$$
and \eqref{eq:mlog} follows by applying the bounded linear operator $I-A$.
\end{proof}

\begin{rem}\begin{enumerate}[(a)]
\item The unquantified version of the above result, namely that $\|AT(t)\|\to0$ as $t\to\infty$ when $T$ is bounded and $\sigma(A)\cap i\RR=\{0\}$ is shown in the more general setting of eventually differentiable semigroups in \cite[Theorem~3.10]{ArPr92}. The result can also be deduced from the Katznelson-Tzafriri theorem. Indeed, it was shown that in the proof of Proposition~\ref{stable_sum} that $\|T(t)AR(1,A)^2\|\to0$ as $t\to\infty$, from which the claim  follows easily; see also \cite[Remark~6.3]{BatChi14}
\item As is shown in \cite[Corollary~2.12]{ChiSei15}, the result in fact holds more generally for bounded semigroups whose generator is not necessarily bounded. When $X$ is a Hilbert space, it follows from \cite[Theorem~5.4]{BaBlSr03} that the condition $\zeta(T)<0$ can be replaced by the condition $\sup_{|s|\ge1}\|R(is,A)\|<\infty$. A more direct way of showing that $\zeta(T)=-\infty$ when the semigroup $T$ has bounded generator is to observe that in this case 
$$T(t)=\sum_{n^=0}^\infty  \frac{t^n}{n!}A^n,\quad t\ge0,$$
with the sum converging in operator norm. In particular, $T$ itself extends to an analytic and exponentially bounded operator-valued family on a sector containing $(0,\infty)$, and the claim follows from the definition of $\zeta(T)$. See \cite{BaBlSr03} for details on the non-analytic growth bound $\zeta(T)$.

\item Theorem~\ref{mlog_gen} can also be deduced from \cite[Proposition~3.1]{Mar11} with the function $M:[1,\infty)\to(0,\infty)$ taken to be constant, since in this case both the $\Mlog^{-1}$-term and the $t^{-1}$-term are dominated by the $\mlog^{-1}$-term.
\end{enumerate}
\end{rem}

The next result is a special case of Theorem~\ref{mlog_gen} dealing with the case of polynomial resolvent growth, and it contains a sharper estimate in the Hilbert space setting.

\begin{thm}\label{thm:poly}
Let $X$ be a complex Banach space and suppose $T$ is a uniformly bounded $C_0$-semigroup on $X$ whose generator $A\in\B(X)$ satisfies $\sigma(A)\cap i\RR=\{0\}$ and $\|R(is,A)\|=O(|s|^{-\alpha})$ as $|s|\to0$ for some $\alpha\ge1$. Then
\begin{equation*}\label{log_decay}
\|AT(t)\|=O\left(\bigg(\frac{\log t}{t}\bigg)^{1/\alpha}\right),\quad t\to\infty.
\end{equation*}
Moreover, if $X$ is a Hilbert space then the logarithm can be omitted.
\end{thm}

\begin{proof}
The first statement is a consequence of Theorem~\ref{mlog_gen} with the choice $m(r)=Cr^{-\alpha}$, $0<r\le 1$, for a suitable constant $C\ge1$, since in this case 
$$\mlog^{-1}(ct)=O\left(\bigg(\frac{\log t}{t}\bigg)^{1/\alpha}\right),\quad t\to\infty,$$
for all $c\in(0,1)$. The second statement is a direct consequence of \cite[Theorem~7.6]{BatChi14} and boundedness of the generator $A$.
\end{proof}

\begin{rem}\label{rem:optimality}
It follows from \cite[Corollary~6.11]{BatChi14} that if in Theorem~\ref{thm:poly} we in fact have $\|R(is,A)\|\asymp|s|^{-\alpha}$ as $|s|\to0$, then there exists a constant $c>0$ such that 
$$\|AT(t)\|\ge\frac{c}{t^{1/\alpha}},\quad t\ge1,$$
provided that $\|sR(is,A)\|\to\infty$ as $|s|\to0$. Since the resolvent growth parameter $n_\phi$ is always strictly greater than 1 by Lemma~\ref{lem:even}, it follows from Proposition~\ref{prp:Rnormestimates} that the latter condition is always satisfied in Theorem~\ref{cor:gen}. 
\end{rem}

We now combine the previous results in this section in order to obtain a general result about the asymptotics of semigroups whose generators are suitable bounded operators. Recall from \eqref{stable_sg} that
$$Y=\left\{x\in X:\lim_{t\to\infty}T(t)x\;\mbox{exists}\right\}.$$

\begin{thm}\label{erg_quant}
Let $T$ be a uniformly bounded $C_0$-semigroup on a space $X$ which is the dual of a complex Banach space $X_*$. Suppose that the generator $A$ of $T$ satisfies $A=B-C$, where $B,C\in\B(X)$, $C$ is invertible, the operator $Q=C^{-1}B$ is power-bounded and satisfies $Q=U'$ for some $U\in\B(X_*)$. Suppose furthermore that  $\sigma(A)\cap i\RR=\{0\}$ and that
$$\|R(is,A)\|\le m(|s|),\quad 0<|s|\le1,$$
for some continuous non-increasing function $m:(0,1]\to[1,\infty)$.

  Then, given $x\in X$, we have $x\in Y$ if and only if there exists $y\in\Fix(Q)$ such that
\eqn{
\label{lim_ergodic}
\bigg\|\frac{1}{n}\sum_{k=1}^nQ^kC^{-1}(x-y)\bigg\|\to0,\quad n\to\infty,
}
 and if \eqref{lim_ergodic} holds then  $T(t)x\to y$ as $t\to\infty$. Moreover, if the convergence in \eqref{lim_ergodic}  is like $O(n^{-1})$ as $n\to\infty$, then for each $c\in(0,1)$
\eqn{
\label{T_mlog}
\|T(t)x-y\|=O\big(\mlog^{-1}(cn)\big),\quad t\to\infty,
}
 where $\mlog$ is as defined in \eqref{eq:mlogdef}. In particular, if  $\|R(is,A)\|=O(|s|^{-\alpha})$ for some $\alpha\ge1$ as $|s|\to0$, then 
\eqn{
\label{T_log}
\|T(t)x-y\|=O\left(\bigg(\frac{\log t}{t}\bigg)^{1/\alpha}\right),\quad t\to\infty,
}
 and the logarithm can be omitted if $X$ is a Hilbert space. 
  \end{thm}

\begin{proof}
The description of the set $Y$ follows immediately by combining Propositions~\ref{stable_sum} and \ref{ergodic}. If  the convergence in \eqref{lim_ergodic} is like $O(n^{-1})$ as $n\to\infty$, then by Proposition~\ref{ergodic} we have that $x-y=Az$ for some $z\in X$, and hence 
$$T(t)x-y=T(t)(x-y)=T(t)Az,\quad t\ge0.$$ 
Thus \eqref{T_mlog} follows from Theorem~\ref{mlog_gen}, and \eqref{T_log} and the statement after it follow from Theorem~\ref{thm:poly}.
\end{proof}

\begin{rem}
By Remark~\ref{sublinear} it is possible to obtain the  unquantified statements of Theorem~\ref{erg_quant} under weaker assumptions.
\end{rem} 

We now come to the proof of Theorem~\ref{cor:gen}.

\begin{proof}[Proof of Theorem~\ref{cor:gen}]
Note first that $X$ has a predual $X_*$ for each choice of $p$. Indeed, if $1<p\le\infty$ then $X$ is the dual of $X_*=\ell^q(\ZZ;\CC^m)$, where $q$ is the H\"older conjugate of $p$, and if $p=1$ then $X$ is the dual of $X_*=c_0(\ZZ;\CC^m)$, the space of $\CC^m$-valued sequences $(x_k)_{k\in\ZZ}$ such that $|x_k|\to0$ as $k\to\pm\infty$. 
 Since $\sigma(A_0)$ is contained in the open left half-plane by assumption (A3), $A_0$ is invertible and hence so is $M_0$.  Moreover, $\sigma(A)\cap i\RR=\{0\}$ by assumption (A4) and Theorem~\ref{thm:Aspec}, and by Proposition~\ref{prp:Rnormestimates} we have that $\|R(is,A)\|\asymp|s|^{-n_\phi}$ as $|s|\to0$.   For $j=0,1$ let $M_j\in\B(X)$ denote the operator given by $M_j(x_k)=(A_jx_k)$, noting that both $M_0$ and $M_1$ commute with the right-shift  operator $S$ on $X$. Moreover,  let $M,N\in\B(X)$ be given by $M=M_1(-M_0)^{-1}$ and $N=(-M_0)^{-1}M_1$  so that $Mx=(A_1R_0 x_k)_{k\in\ZZ}$ and $Nx=( R_0A_1 x_k)_{k\in\ZZ}$ for all $x=(x_k)_{k\in\ZZ}\in X$. Then $A=B-C$ with $B=SM_1$ and $C=-M_0$. Let $Q=C^{-1}B$. Then $Q=SN$ and in particular $Q=U'$, where $U\in\B(X_*)$ is given by $Ux=(A_1^TR_0^T x_{k+1})_{k\in\ZZ}$. Moreover, for $n\ge1$, it follows from our assumption on the matrices $A_0$, $A_1$ that $N^n=\phi(0)^{n-1}N$, and hence $Q^n=\phi(0)^{n-1}S^nN$. Note also that $|\phi(0)|=1$ since $0\in\Omega_\phi$. In particular,  
$$\|Q^n\|=\|S^nN\|\le\|N\|,\quad n\ge1,$$
so $Q$ is power-bounded. Moreover, 
$$Q^nC^{-1}=(-M_0)^{-1}\phi(0)^{n-1}S^nM,\quad n\ge1.$$
Suppose that $1\le p<\infty$ and let $x_0\in X$. Then $\ker(A)=\Fix(Q)=\{0\}$, and since $M_0$ is an isomorphism and $|\phi(0)|=1$, it follows from Theorem~\ref{erg_quant} that $x_0\in Y$ if and only if \eqref{lim_model_finite} holds, and that $x(t)\to0$ as $t\to\infty$ whenever this is the case. If $p=\infty$ then by Theorem~\ref{thm:Aspec} any  $z\in\ker(A)$ has the form $z=(\phi(0)^kz_0)$ for some $z_0\in\Ran(A_0^{-1}A_1)$. For such a $z\in\ker(A)$ let $y=Mz$. Then $y=(\phi(0)^ky_0)$, where $y_0=A_1A_0^{-1}z_0$. In particular, $y_0\in\Ran(A_1)$ and $z_0=Ly_0$. Moreover, $\phi(0)^nS^nMz=y$ for all $n\ge1$ and hence, given $x_0\in X$, Theorem~\ref{erg_quant} implies that $x_0\in Y$ if and only if \eqref{lim_model_infty} holds for some $y_0\in\Ran(A_1)$, and that $x(t)\to z$ as $t\to\infty$ whenever this is the case. When $p=1$ and when $p=\infty$, it is straightforward to see that \eqref{lim_model_finite} and, respectively, \eqref{lim_model_infty} are not satisfied for all $x_0\in X$, whereas \eqref{lim_model_finite} does hold for all $x_0\in X$ when $1<p<\infty$, as can be seen by considering  the dense subspace of finitely supported sequences. Thus $Y=X$ if and only if $1<p<\infty$ and part~(a) is established. For part~(b) note that~(ii) follows immediately from Theorem~\ref{erg_quant}, while if $1\le p<\infty$ and convergence in \eqref{lim_model_finite} is like $O(n^{-1})$ as $n\to\infty$,  Theorem~\ref{erg_quant} shows that 
$$\|x(t)\|=O\left(\bigg(\frac{\log t}{t}\bigg)^{{1/n_\phi}}\right),\quad t\to\infty,$$
and that the logarithm can be omitted when $p=2$. The estimate in \eqref{model_log_finite} now follows by appealing to the  Riesz-Thorin theorem \cite[Theorem~9.3.3]{Gar07} to interpolate these bounds for $1<p<2$ and $2<p<\infty$. Part~(c) follows similarly using the fact that $\dot{x}(t)=AT(t)x_0$ for all $x_0\in X$ and $t\ge0$.
\end{proof}

\begin{rem}\label{rem:gen}
\begin{enumerate}[(a)]
\item The statement in part~(a)(i) can also be deduced from the well-known Arendt-Batty-Lyubich-V\~{u} theorem; see \cite{AreBat88, LyuVu88}. Indeed, the semigroup $T$ is uniformly bounded by assumption (A5), and by Theorem~\ref{thm:Aspec} the other assumptions ensure that the generator $A$ of $T$ has no residual spectrum on the imaginary axis. This argument can be extended to obtain strong stability of $T$ also in the case where $\Omega_\phi$ meets the imaginary axis in several (but necessarily at most finitely many) points.
\item \label{rem:model_optimality}
It follows from Remark~\ref{rem:optimality} together with Lemma~\ref{lem:even} and an application of the uniform boundedness principle that the rates in Theorem~\ref{cor:gen} are optimal when $p=2$ and worse than optimal by at most a logarithmic term when $p\ne 2$. We expect that the quantified statements in Theorem~\ref{cor:gen} remain true without the logarithms even when $p\ne2$, but we leave it as an open problem whether this is indeed the case; see also Remark~\ref*{rem:plat}\eqref{rem:plat_log} and Theorem~\ref{cor:rob} below.
\end{enumerate}
\end{rem}

\section{The platoon model}
\label{sec:plat}

In this section we study a linearised model of an infinitely long platoon of vehicles. The objective is to drive the solution of the system to a configuration in which all of the vehicles are moving at a given constant velocity $v\in\C$ and the separation between the vehicles $k$ and $k-1$ is equal to $c_k\in\C$, $k\in\ZZ$.
For $k\in\ZZ$ and $t\ge0$, we write $d_k(t)$ for the separation between vehicles $k$ and $k-1$ at time $t$, $v_k(t)$ for the velocity of vehicle $k$ at time $t$ and  $a_k(t)$ for the acceleration of vehicle $k$ at time $t$. Furthermore, we let $y_k(t)=c_k-d_k(t)$ denote the deviation of the actual separation from the target separation of vehicles $k$ and $k-1$ at time $t$, and we similarly let $w_k(t)=v_k(t)-v$ stand for the excess velocity of vehicle $k$ at time $t$. Note in particular that, as the variables are allowed to be complex, they can be used to describe the dynamics of the vehicles in the complex plane and not just along a straight line. 
On the other hand, if all the variables  are constrained to be real, the same model can be used to study the behaviour of an infinitely long  chain of vehicles.

As the basis of our study we consider a linear 
model which has been used to study infinitely long chains of cars on a highway
in~\cite{PloSch11,PloWou14,ZwaFir13}, namely
\eqn{
\label{eq:platoon}
\pmat{\dot{y}_k(t)\\\dot{w}_k(t)\\\dot{a}_k(t)} &= \pmat{w_k(t)-w_{k-1}(t)\\a_k(t)\\-\tau\inv a_k(t)+\tau\inv u_k(t)},\quad k\in\ZZ,\;t\ge0,
}
where $\tau>0$ is a parameter and $u_k(t)$ is the control input of vehicle $k$.
In the above references 
the model~\eqref{eq:platoon} was studied on the space $X=\lp[2](\C^3)$, and it has in particular been shown that the system is not exponentially stabilisable~\cite{JovBam05,ZwaFir13} but
that strong stability can be achieved~\cite{CurIft09,JovBam05}. In this paper we study model~\eqref{eq:platoon} on the spaces $X=\lp[p](\C^3)$ for $1\le p\le\infty$, and in particular we include the case $p=\infty$ argued in~\cite{JovBam05} to be the most realistic.

We begin by rewriting the problem in the form of \eqref{eq:ODEintro} for the state vectors
$$x_k(t) = \pmat{y_k(t)\\w_k(t)\\a_k(t)},\quad k\in\ZZ,\;t\ge0,$$
by applying an identical state feedback 
$$u_k(t) = \beta_1y_k(t) + \beta_2w_k(t) + \beta_3a_k(t),\quad k\in\ZZ,\;t\ge0,$$
to each of the vehicles, where $\beta_1,\beta_2,\beta_3\in\CC$ are constants. This control law requires that the state vectors $x_k(t)$ are known and available for feedback. Equations~\eqref{eq:platoon} can then be written in the form~\eqref{eq:ODEintro} with matrices
\eq{
\Ad
=\pmat{0&1&0\\0&0&1\\-\ga_0&-\ga_1&-\ga_2}
\qquad\mbox{and}\qquad
\Aod=\pmat{0&-1&0\\0&0&0\\0&0&0} 
}
where $\ga_0= -\beta_1/\tau$, $\ga_1 = - \beta_2/\tau$, and $\ga_2 = (1-\beta_3)/\tau$ can be freely assigned by choosing appropriate feedback parameters $\beta_1,\beta_2,\beta_3\in\C$. This in turn allows us to choose the eigenvalues of the matrix $\Aod$. Since $\rank A_1=1$, we know from Remark~\ref{rem:ass1} that conditions (A1) and (A2) of Assumption~\ref{ass:PScond} are satisfied, and the characteristic function $\phi$ is given by the formula
\eq{
\phi(\gl) 
= \frac{\ga_0}{p(\gl)}, \qquad \gl\in\CC\setminus \gs(\Ad),
}
where $p(\gl)=\gl^3+\ga_2\gl^2+\ga_1\gl+\ga_0$ is the characteristic polynomial of $\Ad$. 
Note that $\phi(0)=1$ and hence $0\in\gs(A)$ by Theorem~\ref{thm:Aspec}. It follows that the platoon system cannot be stabilised exponentially.
Our main goal is to choose the parameters $\ga_0,\ga_1,\ga_2\in\C$ in such a way that the platoon system achieves good stability properties. 
The simplest possible characteristic polynomial is $p(\lambda)=(\lambda-\lambda_0)^3$ corresponding to the choices 
$\alpha_0=-\lambda_0^3$, $\alpha_1=3\lambda_0^2$, and $\alpha_2=-3\lambda_0$ for a fixed $\lambda_0\in\CC$. In this case 
$$\Omega_\phi=\big\{\lambda\in\CC:|\lambda-\lambda_0|=|\lambda_0|\big\},$$
 so in order for  conditions
(A3) and (A4) of Assumptions~\ref{ass:further} to be satisfied, so that $\gs(A)\subset \C_-\cup \set{0}$, it is necessary to choose  $\lambda_0=-\zeta$ for some $\zeta>0$.  
It is possible in principle to derive more general necessary geometric conditions on the roots of $p$ which ensure that (A3) and (A4) are satisfied. We restrict ourselves here to exhibiting, in Figure~\ref{fig:Platspec},  several examples of level sets $\Omega_\phi$ and spectra $\gs(A)$ for different choices of the parameters $\ga_0,\ga_1,\ga_2\in\C_-$.

\begin{figure}[ht]
\hspace{-.8cm}
    \begin{minipage}[c]{.19\linewidth}
	\includegraphics[width=1.2\linewidth]{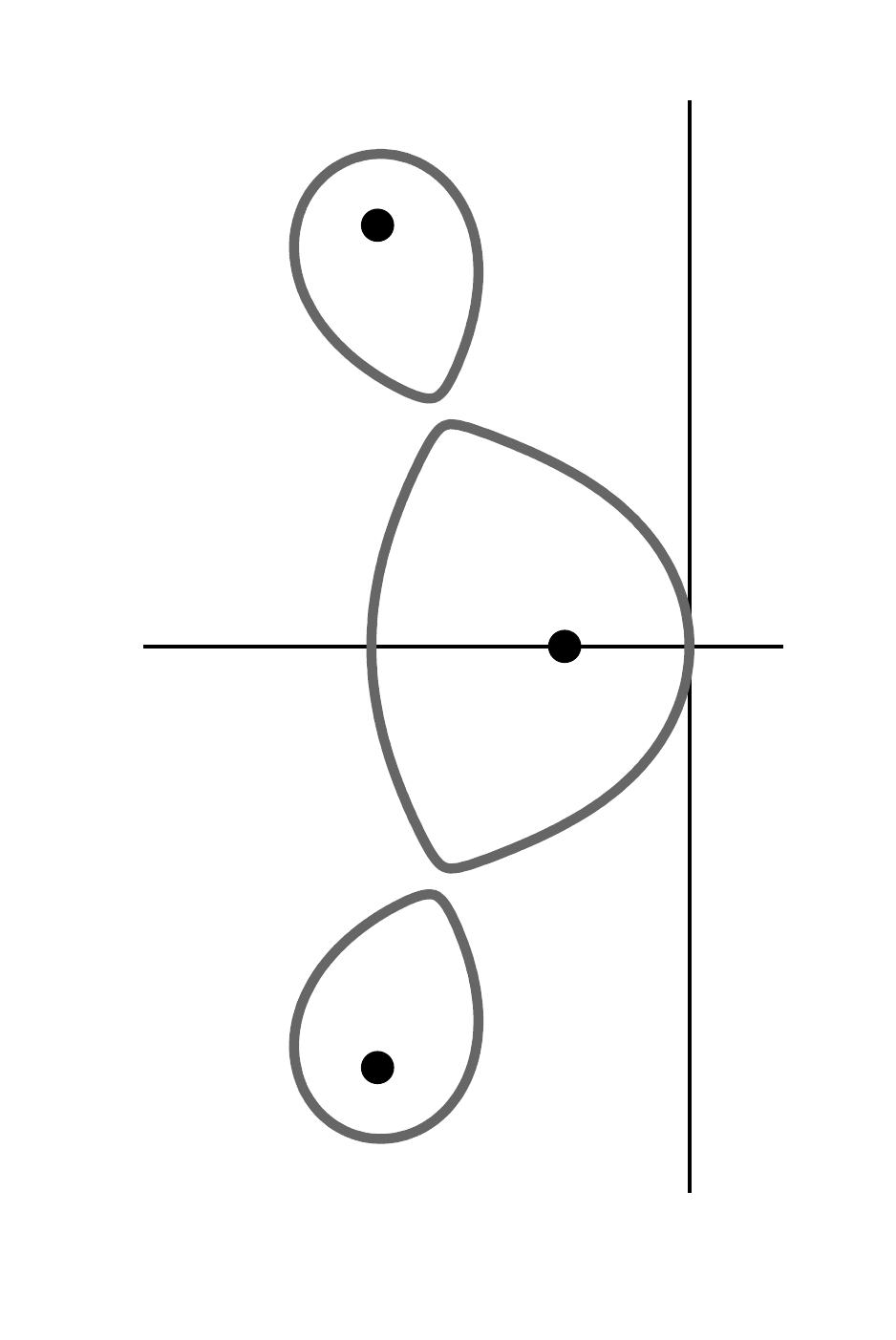}
    \end{minipage}
    \begin{minipage}[c]{.18\linewidth}
      \begin{center}
	\includegraphics[width=1.2\linewidth]{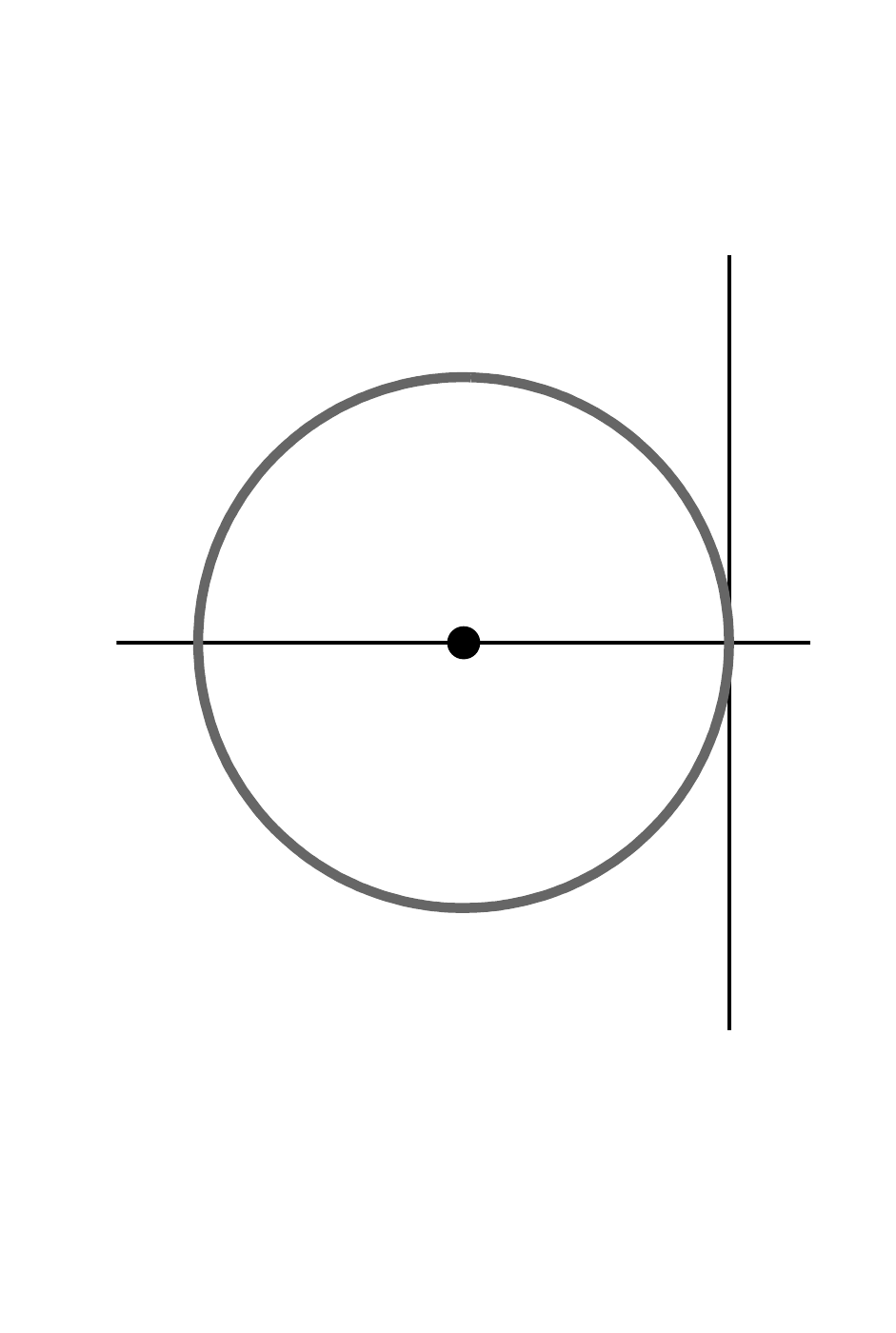}
      \end{center}
    \end{minipage}
    \begin{minipage}[c]{.19\linewidth}
      \begin{center}
	\includegraphics[width=1.1\linewidth]{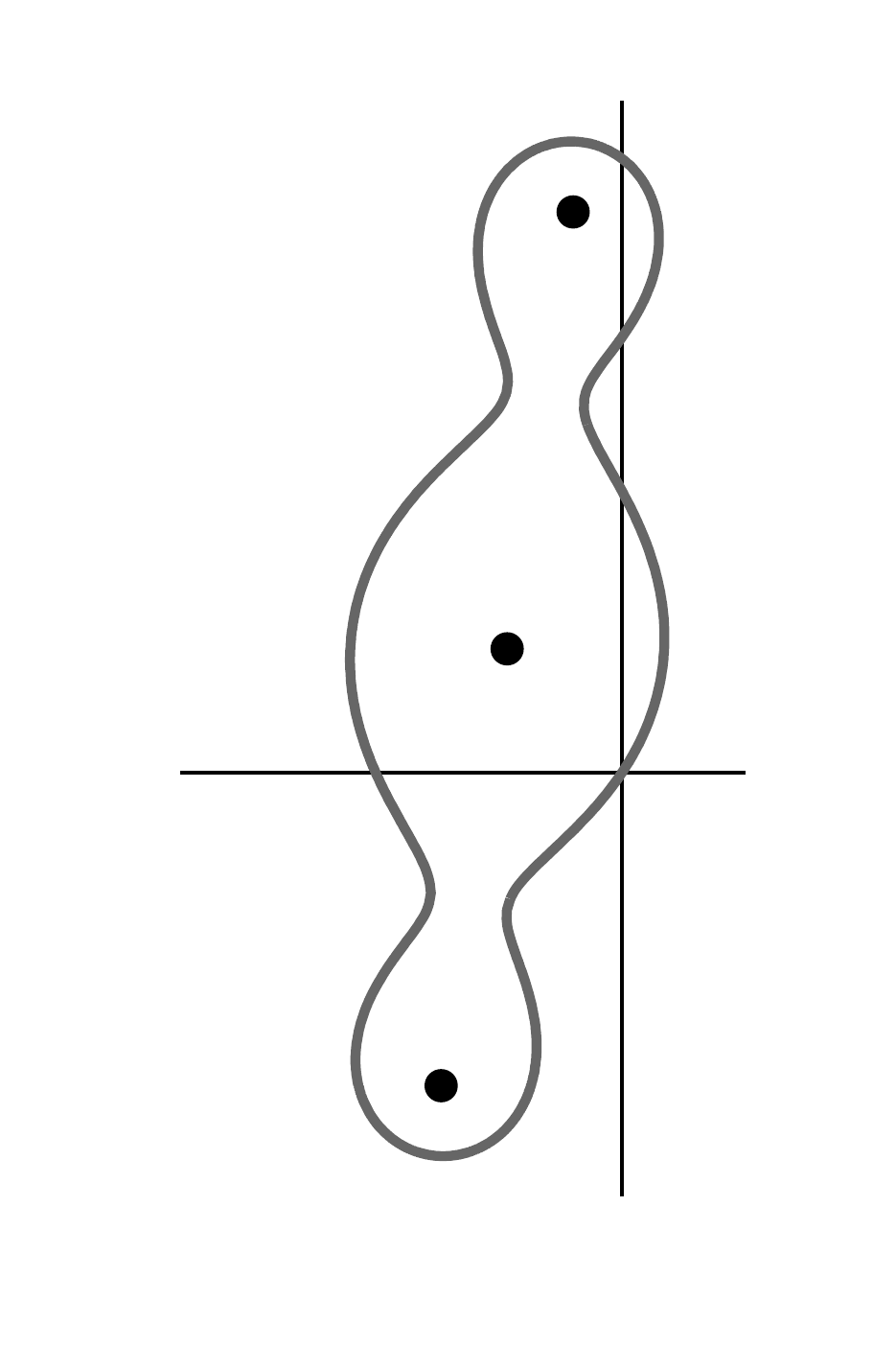}
      \end{center}
    \end{minipage} 
    \begin{minipage}[c]{.17\linewidth}
      \begin{center}
	\includegraphics[width=1.2\linewidth]{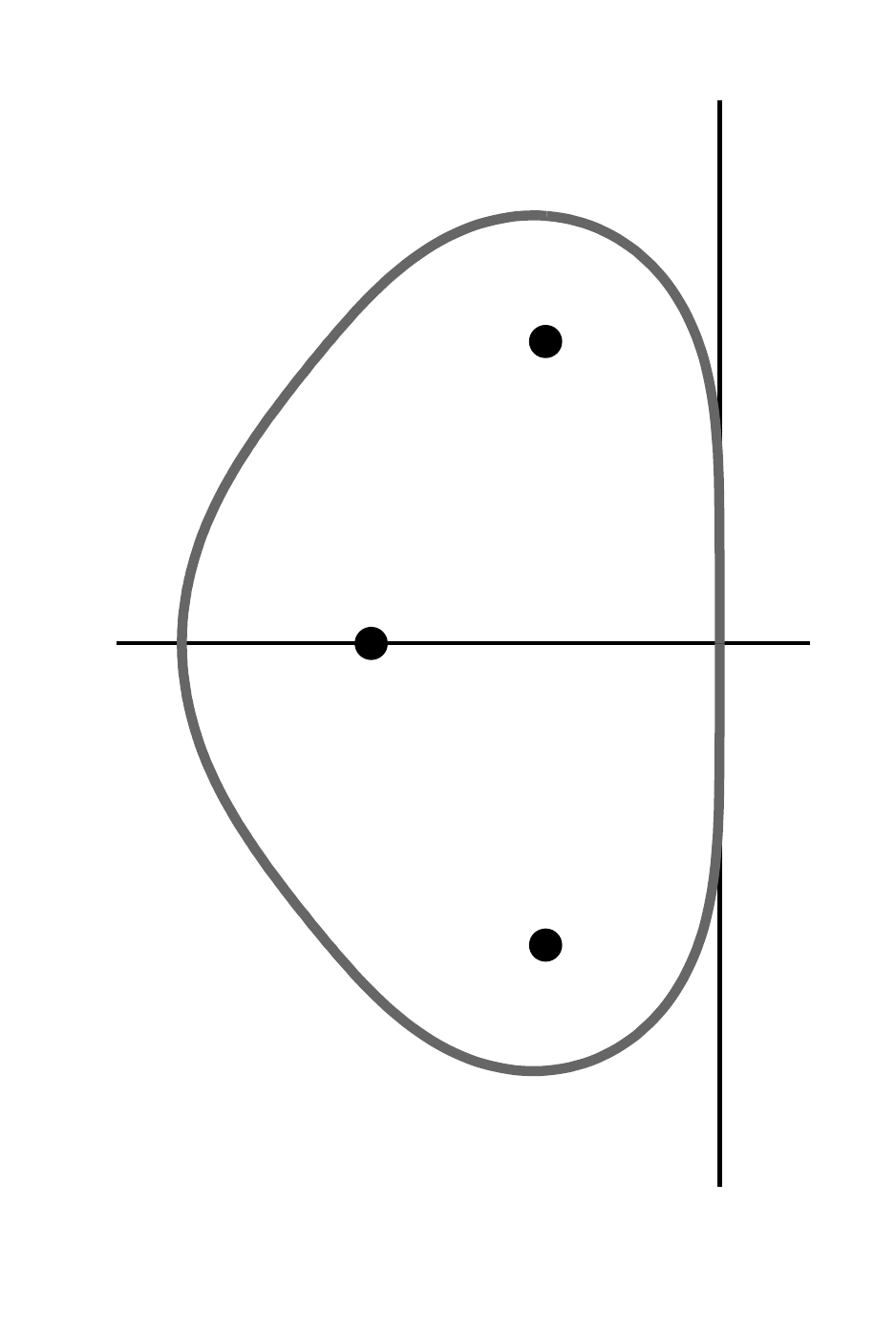}
      \end{center}
    \end{minipage}
    \begin{minipage}[c]{.21\linewidth}
      \begin{center}
	\includegraphics[width=1.25\linewidth]{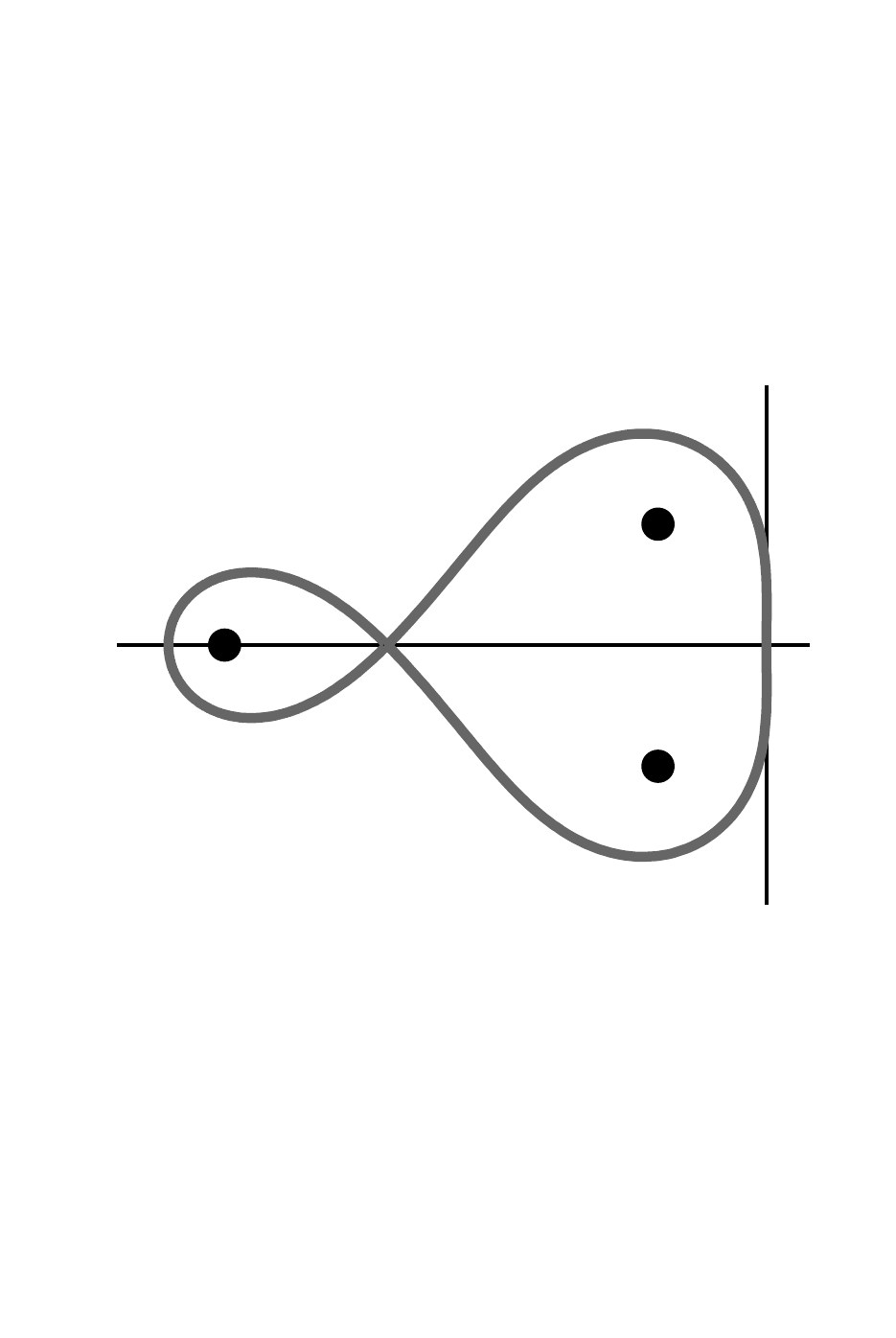}
      \end{center}
    \end{minipage} 
  \caption{The level set $\Omega_\phi$ and $\gs(\Ad)$ for various choices of $\Ad$.}
  \label{fig:Platspec}
\end{figure}

We now consider the Cauchy problem \eqref{CP} for the platoon problem with a characteristic function $\phi$ having just one pole. Our main asymptotic result is a consequence of the general results proved in the preceding sections. Recall that for $1\le p\le\infty$ we denote the set of initial states $x(0)=x_0$ leading to convergent solutions $x(t)$ of the platoon system by
$$Y=\left\{x_0\in X:\lim_{t\to\infty} x(t)\;\mbox{exists}\right\}.$$

\begin{thm}\label{cor:plat}
Let $1\le p\le\infty$ and consider the platoon model with the choices $\alpha_0=\zeta^3$, $\alpha_1=3\zeta^2$ and $\alpha_2=3\zeta$, where $\zeta>0$ is a fixed real number.  
\begin{enumerate}
\item[\textup{(a)}] We have $Y=X$ if and only if $1< p<\infty$. More specifically:
\begin{enumerate}
\item[\textup{(i)}] If $1<p<\infty$ then $Y=X$ and $x(t)\to0$ for all $x_0\in X$.

\item[\textup{(ii)}] If $p=1$ and $x_0\in X$ then $x_0\in Y$ if and only if the vector $y_0=(y_k(0))_{k\in\ZZ}\in\ell^1(\ZZ)$ of initial deviations is such that
\eqn{
\label{lim_plat_finite}
\bigg\|\frac{1}{n}\sum_{k=1}^nS^{k}y_0\bigg\|_{\ell^1(\ZZ)}\!\!\to0,\quad n\to\infty,
}
and if this  holds then $x(t)\to0$ as $t\to\infty$. 
\item[\textup{(iii)}] If $p=\infty$ and $x_0\in X$ then $x_0\in Y$ if and only if there exists $c\in\CC$ such that for $y=(\dotsc,c,c,c,\dotsc)$ we have
\eqn{
\label{lim_plat_infty}
\bigg\|\frac{1}{n}\sum_{k=1}^nS^{k}y_0-y\bigg\|_{\ell^\infty(\ZZ)}\!\!\to0,\quad n\to\infty,
}
and if this holds then $x(t)\to z$ as $t\to\infty$, where 
\eq{
z=\left(\dotsc,\begin{pmatrix}
c\\ -\zeta c/3\\ 0
    \end{pmatrix},\begin{pmatrix}
c\\ -\zeta c/3\\ 0
    \end{pmatrix},\begin{pmatrix}
c\\ -\zeta c/3\\ 0
    \end{pmatrix},\dotsc\right).
} 
\end{enumerate}

\item[\textup{(b)}] 
\begin{enumerate}
\item[\textup{(i)}] If $1\le p<\infty$ and the decay in \eqref{lim_plat_finite} is like $O(n^{-1})$ as $n\to\infty$  then
\eq{
\|x(t)\|=O\left(\bigg(\frac{(\log t)^{|1-2/p|}}{t}\bigg)^{1/{2}}\right),\quad t\to\infty.
}
\item[\textup{(i)}] If $p=\infty$ and the decay in \eqref{lim_plat_infty} is like $O(n^{-1})$ as $n\to\infty$ then
\eq{
\|x(t)-z\|=O\bigg(\bigg(\frac{\log t}{t}\bigg)^{1/{2}}\bigg),\quad t\to\infty.
}
\end{enumerate}
\item[\textup{(c)}] For $1\le p\le\infty$ and all  $x_0\in X$ we have
 \eq{
\|\dot{x}(t)\|=O\left(\bigg(\frac{(\log t)^{|1-2/p|}}{t}\bigg)^{1/{2}}\right),\quad t\to\infty.
} 
 \end{enumerate}
\end{thm}

\begin{proof}
 Note that (A1) holds and that (A2) is satisfied for the function
 $$\phi(\lambda)=\frac{\zeta^3}{(\lambda+\zeta)^3},\quad \lambda\ne-\zeta.$$
  As above, we have $\gs(\Ad)=\set{-\dr}$, so that (A3) holds, and since 
  $$\Omega_\phi=\{\lambda\in\CC:|\lambda+\zeta|=\zeta\},$$ 
  we see that (A4) holds as well. Furthermore, (A5) holds by Lemma~\ref{lem:repeated} and the first part of Theorem~\ref{thm:Aunifbdd}.
A simple calculation based on the ideas used in Lemma~\ref{lem:even} shows that $n_\phi=2$. Noting that $\phi(0)=1$ and that
$$A_1(-A_0)^{-1}=\pmat{1&0&0\\0&0&0\\0&0&0},\quad\quad (-A_0)^{-1}A_1=\pmat{0&-3/\zeta&0\\0&1&0\\0&0&0},$$
the result follows  from Theorem~\ref{cor:gen}.
\end{proof}

\begin{rem}\label{rem:plat}
\begin{enumerate}[(a)]
\item \label{rem:plat_log} We do not know whether the logarithms in the decay estimates of Theorem~\ref{cor:plat} are needed when $p\ne2$. We suspect not; see also Remark~\ref*{rem:gen}\eqref{rem:model_optimality} above and Theorem~\ref{cor:rob} below.
\item  \label{prp:platoonnoncontr} It follows from straightforward estimates that the semigroup $T$ generated by the operator $A$ in the platoon model is  in general not contractive, even when $p=2$.
\item Note that the above analysis can also be used in the setting considered in \cite{PloWou14}, where the objective of attaining given target separations as $t\to\infty$  is replaced by the objective that the separations should approach $c_{k}+h v_k(t)$, where $h>0$  and $c_k\in\CC$ are  constants and $v_k(t)$ is the velocity of vehicle $k\in\ZZ$ at time $t\ge0$.
\end{enumerate}
\end{rem}

\section{The robot rendezvous problem}
\label{sec:robots}

We now return to the robot rendezvous problem, which corresponds in the general setting of \eqref{eq:ODEintro} to the choices $m=1$, $A_0=-1$ and $A_1=1$. In particular, the Banach space we are working in is $X=\ell^p(\ZZ)$, where $1\le p\le \infty$. The following result, which can be viewed as an extension of the results in \cite{FeiFra12}, is in large part a consequence of Theorem~\ref{cor:gen} but with slightly sharper estimates on the rates of decay. For  $1\le p\le\infty$, we once again use the notation
$$Y=\left\{x_0\in X:\lim_{t\to\infty} x(t)\;\mbox{exists}\right\},$$
where $x(t)$, $t\ge0$, now denotes the solution of the robot rendezvous problem with initial condition $x(0)=x_0$.

\begin{thm}\label{cor:rob}
Let $1\le p\le\infty$ and consider the robot rendezvous problem. 
\begin{enumerate}
\item[\textup{(a)}] We have $Y=X$ if and only if $1< p<\infty$. More specifically:
\begin{enumerate}
\item[\textup{(i)}] If $1<p<\infty$ then $Y=X$ and $x(t)\to0$ for all $x_0\in X$.

\item[\textup{(ii)}] If $p=1$ and $x_0\in X$ then $x_0\in Y$ if and only if 
\eqn{
\label{lim_rob_finite}
\bigg\|\frac{1}{n}\sum_{k=1}^nS^{k}x_0\bigg\|\to0,\quad n\to\infty,
}
and if this  holds then $x(t)\to0$ as $t\to\infty$. 
\item[\textup{(iii)}] If $p=\infty$ and $x_0\in X$ then $x_0\in Y$ if and  only if there exists a constant sequence $z\in X$ such that
\eqn{
\label{lim_rob_infty}
\bigg\|\frac{1}{n}\sum_{k=1}^nS^{k}x_0-z\bigg\|\to0,\quad n\to\infty,
}
and if this holds then $x(t)\to z$ as $t\to\infty$.
\end{enumerate}

\item[\textup{(b)}] 
\begin{enumerate}
\item[\textup{(i)}] If $1\le p<\infty$ and the decay in \eqref{lim_rob_finite} is like $O(n^{-1})$ as $n\to\infty$  then
\eqn{
\label{rob_log_finite}
\|x(t)\|=O\big(t^{-1/2}\big),\quad t\to\infty,
}
\item[\textup{(i)}] If $p=\infty$ and the decay in \eqref{lim_rob_infty} is like $O(n^{-1})$ as $n\to\infty$ then
\eqn{
\label{rob_log_infty}
\|x(t)-z\|=O\big(t^{-1/2}\big),\quad t\to\infty.
}
\end{enumerate}
\item[\textup{(c)}] 
For $1\le p\le\infty$ and all  $x_0\in X$ we have
 \eqn{
\label{deriv_rob_log}
\|\dot{x}(t)\|=O\big(t^{-1/2}\big),\quad t\to\infty.
} 
\end{enumerate}
 Finally, the rate $t^{-1/2}$ in \eqref{rob_log_finite}, \eqref{rob_log_infty} and \eqref{deriv_rob_log} is optimal.

\end{thm}

\begin{proof}
Note that (A1) holds and that (A2) is satisfied for the function
 $$\phi(\lambda)=\frac{1}{\lambda+1},\quad \lambda\ne-1.$$
We also have $\gs(\Ad)=\set{-1}$, so that (A3) holds, and since 
  $$\Omega_\phi=\{\lambda\in\CC:|\lambda+1|=1\},$$ 
  we see that (A4) holds as well. Assumption (A5) again holds by Lemma~\ref{lem:repeated} and Theorem~\ref{thm:Aunifbdd}, and indeed the second part of the latter result even shows that the semigroup is contractive.
As in the proof of Theorem~\ref{cor:plat}, a simple calculation shows that $n_\phi=2$, so all of the statements follow from Theorem~\ref{cor:gen} except for the rates in equations  \eqref{rob_log_finite}, \eqref{rob_log_infty} and \eqref{deriv_rob_log} and the final statement concerning optimality. The latter follows as in Remark~\ref*{rem:gen}\eqref{rem:model_optimality}. In order to obtain the sharper rates we require a better estimate on the asymptotic behaviour of $\|AT(t)\|$ as $t\to\infty$ than is given in  Theorem~\ref{thm:poly} for the general case.

For $t\ge0$, let $y(t)\in \ell^1(\ZZ)$ be the scalar-valued sequence whose $k$-th term is given by
$$y_k(t)=\frac{t^k}{k!}-\frac{t^{k+1}}{(k+1)!},\quad k\ge0,$$
and $y_k(t)=0$ for $k<0$. It is shown in the proof of \cite[Theorem~3]{FeiFra12} that, given $x_0\in X$, 
$$AT(t)x_0=e^{-t}\big(y(t)*z_0-x_0\big),\quad t\ge0,$$
 where $z_0=Sx_0$ with $S$ being the right-shift. Then $\|z_0\|=\|x_0\|$ and it follows from Young's inequality that 
 $$\|AT(t)x_0\|\le e^{-t}\big(1+\|y(t)\|_{\ell^1(\ZZ)}\big)\|x_0\|,\quad t\ge0.$$
 In particular,
 $$\|AT(t)\|\le e^{-t}\big(1+\|y(t)\|_{\ell^1(\ZZ)}\big)\quad t\ge0.$$
 As explained in the proof of \cite[Theorem~3]{FeiFra12}, for each $t\ge0$ there exists an integer $n(t)\ge0$ such that $n(t)\le t\le n(t)+1$ and 
 $$\|y(t)\|_{\ell^1(\ZZ)}\le 2\frac{t^{n(t)}}{n(t)!}.$$
By Stirling's approximation, 
 $$n(t)!\ge \sqrt{2\pi n(t)}\left(\frac{n(t)}{e}\right)^{n(t)},\quad t\ge1.$$
Now straightforward estimates show that
 \eqn{\label{AT_est}
 \|AT(t)\|\le \frac{C}{t^{1/2}},\quad t\ge2,
 }
 for some $C>0$, and the result follows as in the proof of Theorem~\ref{cor:gen}.
\end{proof}

\begin{rem}
\begin{enumerate}[(a)]
\item Note that in the above setting, the semigroup has an explicit representation. Indeed, if $T(t)(x_k)=(y_k(t))$ for $(x_k),(y_k)\in X$ and $t\ge0$, then
$$y_k(t)=e^{-t}\sum_{n=0}^\infty\frac{t^n}{n!} x_{k-n},\quad k\in\ZZ, \;t\ge0.$$
In particular, an application of Young's inequality gives an alternative, more direct proof of the fact that the semigroup $T$ is contractive in this case for  $1\le p\le\infty$; cf.\ Remark~\ref*{rem:plat}\eqref{prp:platoonnoncontr}.
\item The above proof can be refined to give an explicit constant $C$ in \eqref{AT_est}. For instance, it is straightforward to show that the value 
$$C=\frac{t_0^{1/2}}{e^{t_0}}+\left(\frac{2}{\pi}\right)^{1/2}\left(1-\frac{1}{t_0}\right)^{-t_0-1/2}$$
gives the inequality for the range $t\ge t_0>1$. In particular, $C=4.705$ works for $t\ge2$, $C=2.191$ works for $t\ge100$, and as $t_0\to\infty$ the value of the constant approaches $e(2/\pi)^{1/2}\approx 2.169$.
\item For further discussion of the robot rendezvous problem and in particular its connection with the theory of Borel summability, see \cite{FeiFra12, FeiFra12b}.
\end{enumerate}
\end{rem}

We conclude by briefly considering an interesting generalisation of the robot rendezvous problem in which the original differential equations are replaced by
$$\dot{x}_k(t)=x_{k-1}(t)+\alpha_k x_k(t),\quad k\in\ZZ,\; t\ge0,$$
where for each $k\in\ZZ$ either $\alpha_k=-1$ or $\re\alpha_k<-1$. The original robot rendezvous problem corresponds to the choice $\alpha_k=-1$ for all $k\in\ZZ$. If we again let $X=\ell^p(\CC)$ for $1\le p\le\infty$, we are led to consider the semigroup $T$ generated by the operator $A\in\B(X)$ given by $Ax=(x_{k-1}+\alpha_k x_k)_{k\in\ZZ}$ for all $(x_k)_{k\in\ZZ}\in X$. We restrict ourselves to stating a result about the decay of $\|AT(t)\|$ as $t\to\infty$, which could be used to obtain statements about  orbits and their derivatives as in Sections 4 and 5.

\begin{thm}\label{thm:robot_ext}
In the modified robot rendezvous problem considered above, let $\Omega=\{\alpha_k: k\in\ZZ\}$ and suppose that $-1\in\Omega$. Then, for all $c\in(0,1)$,
$$\|AT(t)\|=O\big(\mlog^{-1}(ct)\big),\ t\to\infty,$$
where $m:(0,1]\to[1,\infty)$ is defined by 
\begin{equation}\label{eq:robot_m}
m(r)=\sup_{r\le|s|\le1}\frac{1}{\dist(i s,\Omega)-1},\quad 0<r\le 1,
\end{equation}
and $\mlog$ is as in Theorem~\textup{\ref{mlog_gen}}.
\end{thm}

\begin{proof}
  Let $\gl\in \C$ be such that $\dist(\gl,\Omega)>1$ and define $R(\gl)\in \B(X)$ by
  \eq{
  R(\gl)x = \left( \sum_{\ell=0}^\infty \prod_{j=0}^\ell \frac{x_{k-\ell}}{\gl-\ga_{k-j}} \right)_{k\in\Z}
  }
  for all $x=(x_k)_{k\in\Z}\in X$. Straightforward computations show that $R(\gl)(\gl-A)x=(\gl-A)R(\gl)x=x$ for all $x\in X$, and hence $\lambda\notin\sigma(A)$ and $R(\gl)=R(\lambda,A)$ for all $\lambda\in\CC$ such that $\dist(\lambda,\Omega)>1$.
  Furthermore,  it is straightforward to verify that if $\dist(\lambda,\Omega)>1$ and $\norm{x}=1$, then
\eq{
\norm{R(\gl,A)x}
\leq \sum_{\ell=0}^\infty \frac{1}{\dist(\gl,\Omega)^{\ell+1}}
= \frac{1}{\dist(\gl,\Omega)-1},
}
 and hence $\norm{R(\gl,A)}\leq (\dist(\gl,\Omega)-1)\inv$ 
for all $\lambda\in\CC$ such that $\dist(\lambda,\Omega)>1$. In particular, $\|R(is,A)\|\le m(|s|)$ for $0<|s|\le1$, where $m:(0,1]\to[1,\infty)$ is as in \eqref{eq:robot_m}, so the result follows from Theorem~\ref{mlog_gen}.
\end{proof}

\begin{rem}\label{rem:robot_ext}
Note that the conclusion of Theorem~\ref{thm:robot_ext} remains true whenever $\Omega$ is replaced by any set $\Omega'$ such that 
$$\{\alpha_k: k\in\ZZ\}\subset\Omega'\subset\{\gl\in\C:\re\gl<-1\}\cup\{-1\}.$$ 
In particular, we can take $\Omega'=\Omega_\psi$, where $\Omega_\psi=\{\lambda\in\CC:\re\lambda\le-\psi(|\im\lambda|)\}$ for some non-decreasing and continuously differentiable function $\psi:[0,1]\to[1,\infty)$ satisfying $\psi(0)=1$ and $\psi(s)>1$ for $s\in(0,1]$; see Figure~\ref{fig:robots}. Then \eqref{eq:robot_m} becomes
\begin{equation}\label{eq:robot_m2}
m(r)=\frac{1}{\dist(i r,\Omega_\psi)-1},\quad 0<r\le 1.
\end{equation}
  If $\psi'(0)>0$, it is easy to see that $m(r)\asymp r^{-2}$ as $r\to0+$ and therefore 
 \begin{equation}\label{eq:robot_decay2}
\|AT(t)\|=O\bigg(\bigg(\frac{\log t}{t}\bigg)^{1/2}\bigg),\quad t\to\infty.
\end{equation}
 On the other hand if $\psi'(0)=0$ it follows from  geometric considerations that 
$$m(r)=\left(\psi(q(r))\big(1+\psi'(q(r))^2\big)^{1/2}-1\right)^{-1},\quad0< r\le 1,$$
where $q:(0,1]\to(0,1]$ is the inverse function of the map $p(r)=r+\psi(r)\psi'(r)$, $0<r\le 1$. 

\begin{figure}
\centering
  \def\svgwidth{200pt}
 { 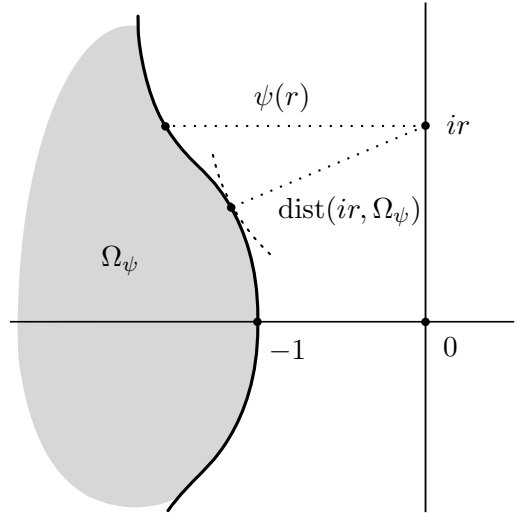}
\caption{The region $\Omega_\psi$.}
\label{fig:robots}
\end{figure}
\end{rem}

\begin{ex}
\begin{enumerate}[(a)]
\item In the original robot rendezvous problem, where $\Omega=\{-1\}$, it follows from  Theorem~\ref{thm:robot_ext} that 
\eqref{eq:robot_decay2} holds. The proof of Theorem~\ref{cor:rob} shows that the logarithm can be omitted.
\item In the context of Remark~\ref{rem:robot_ext}, if $\psi(s)=1+s^\alpha$ for some $\alpha\ge1$, then crude estimates show that $m(r)\asymp r^{-2}$ as $r\to0+$ if $1\le\alpha<2$ and $m(r)\asymp r^{-\alpha}$ as $r\to0+$ if $\alpha\ge2.$ Hence in the first case \eqref{eq:robot_decay2}  holds, while for $\alpha\ge2$  
$$\|AT(t)\|=O\bigg(\bigg(\frac{\log t}{t}\bigg)^{1/\alpha}\bigg),\quad t\to\infty.$$
If $p=2$, so that $X$ is a Hilbert space, it follows from Theorem~\ref{thm:poly} that the logarithm can be omitted in both cases.
\end{enumerate}
\end{ex}

\section{Conclusion}\label{sec:concl}

The main result of this paper, Theorem~\ref{cor:gen}, is a powerful tool for studying the asymptotic behaviour of solutions to a rather general class of infinite systems of coupled differential equations. The versatility of the general theory is illustrated by the applications presented in Sections~\ref{sec:plat} and \ref{sec:robots} to two important special cases: the platoon model and the robot rendezvous problem. Underlying Theorem~\ref{cor:gen} are a number of abstract results from operator theory and in particular the asymptotic theory of operator semigroups. It is  striking how effective the results obtained by these abstract techniques are even in particular examples, shedding new light both on the platoon model and the robot rendezvous problem. Nevertheless, a number of  important questions remain open. The first question, namely whether the logarithmic factors are needed in Theorem~\ref{cor:gen} when $p\ne2$, was already raised in Remark~\ref*{rem:gen}\eqref{rem:model_optimality}. Here it would already be of interest to have an affirmative answer in certain special cases, for instance the platoon model dealt with in Theorem~\ref{cor:plat}; see Remark~\ref*{rem:plat}\eqref{rem:plat_log}. Another aspect of the theory which would benefit from further development is the condition for uniform boundedness of the semigroup presented in Theorem~\ref{thm:Aunifbdd}, since in its present state this condition is rather difficult to verify except for relatively simple characteristic functions. Furthermore, it remains to be determined to what extent the results obtained here can be extended to situations involving more complicated coupling, such as systems in which the evolution of each subsystem depends on the states of several other subsystems rather than just one.  Finally, it would seem worth investigating the corresponding questions in the discrete-time setting, both from an applications perspective and in view of the fact that the corresponding abstract theory is equally well developed as in the continuous-time setting. We hope to address some of these issues in future publications.

\bibliography{robots-reference}
\bibliographystyle{plain}

\end{document}